\documentclass{article}
\usepackage{amssymb}
\usepackage{amsthm}
\usepackage{amsmath}
\usepackage{xspace}
\usepackage{url}
\usepackage{cancel}
\usepackage{stmaryrd}

\newtheorem{theorem}[equation]{Theorem}
 \newtheorem{corollary}[equation]{Corollary}
 \newtheorem{lemma}[equation]{Lemma}
 \newtheorem{proposition}[equation]{Proposition}
 \newtheorem{definition}[equation]{Definition}
 \newtheorem{remark}{\emph{Remark}}
 \newtheorem{example}{\emph{Example}}

\newcommand{\mix}{\textsc{Mix}}
\newcommand{\MLL}{\textbf{MLL}\xspace}
\newcommand{\NP}{\textbf{NP}}
\newcommand{\MLLminus}{\textbf{MLL$^-$}\xspace}
\newcommand{\MLLplusmix}{\MLL+\mix}
\newcommand{\MLLmix}{\MLLplusmix}
\newcommand{\Par}{\mathbin{\bindnasrepma}}

 \newcommand{\rk}{\mathrm{rk}}

\newcommand{\finset}[1]{\{ #1 \}} 

\newcommand{\LK}{\textbf{LK}\xspace}

\newcommand{\W}{\mathsf{W}}

\newcommand{\ox}{\otimes}

\newcommand{\cut}{\mathit{cut}}

\newcommand{\C}{\textsc{C}}

\newcommand{\G}{\mathrm{\Gamma}}
\newcommand{\D}{\mathrm{\Delta}}

\newcommand{\Ax}{\textsc{Ax}}
\newcommand{\Cut}{\textsc{Cut}}

\newcommand{\turnstile}{\vdash}

\usepackage{tikz,lscape}
\usetikzlibrary{calc,chains,arrows,decorations.pathmorphing,positioning,patterns,backgrounds,fit,scopes}
\input prooftree.sty
\input xy
\xyoption{all}
\usepackage{proof}
\title{Canonical proof nets for classical logic}
\author{Richard McKinley}
\begin{document}
\newcommand{\str}{\mathrm{str}}
\newcommand{\var}{}
\newcommand{\LKvar}{\LK^*}
\newcommand{\LKvarnet}{\LK^*_\mathrm{net}}
\newcommand{\Mp}{\textbf{Mp}}
\newcommand{\LKED}{\LK_{ed}}
\newcommand{\LKE}{\LKvar_e}

\newcommand{\meld}{\mathsf{M}}


\maketitle
\begin{abstract}
  Proof nets provide abstract counterparts to sequent proofs modulo
  rule permutations; the idea being that if two proofs have the same
  underlying proof-net, they are in essence the same proof.  Providing
  a convincing proof-net counterpart to proofs in the classical
  sequent calculus is thus an important step in understanding
  classical sequent calculus proofs. By convincing, we mean that (a)
  there should be a canonical function from sequent proofs to proof
  nets, (b) it should be possible to check the correctness of a net in
  polynomial time, (c) every correct net should be obtainable from a
  sequent calculus proof, and (d) there should be a cut-elimination
  procedure which preserves correctness.

  Previous attempts to give proof-net-like objects for propositional
  classical logic have failed at least one of the above conditions.
  In \cite{McK10exnets}, the author presented a calculus of proof nets
  (expansion nets) satisfying (a) and (b); the paper defined a sequent
  calculus corresponding to expansion nets but gave no explicit
  demonstration of (c).  That sequent calculus, called $\LKvar$ in
  this paper, is a novel one-sided sequent calculus with both
  additively and multiplicatively formulated disjunction rules.  In
  this paper (a self-contained extended version of \cite{McK10exnets}) , we give a full proof of (c) for expansion nets with
  respect to $\LKvar$, and in addition give a cut-elimination
  procedure internal to expansion nets -- this makes expansion nets
  the first notion of proof-net for classical logic satisfying all
  four criteria.
\end{abstract}
\section{Introduction}
\label{sec:introduction}
\emph{Proof theory}, the study of formal proofs, was invented as a
tool to study the consistency of mathematical theories, one of
Hilbert's famous 23 problems.  However, Hilbert had originally
considered presenting at his Paris lecture a 24th problem~\cite{Thiele01hilbert}
which concerned proofs directly: he proposed ``develop(ing) a theory
of mathematical proof in general''.  Central to this question is the
idea that usual proofs, as written down by mathematicians, or
formalized in, for example, Gentzen's sequent calculus~\cite{Gentzen34}, are syntactic representations of much more abstract proof
objects. Given that, we should be able to tell when two syntactic
proofs represent the same abstract proof.

It is striking how difficult this question seems to be, even for
propositional classical logic.  In contrast to the well-developed
theory of proof-identity for intuitionistic natural deduction (given
by interpretation of proofs in a cartesian-closed category), the
theory of identity for proofs in classical logic is very poorly
understood.  Investigations by several researchers over the last ten
years~\cite{Rob03ProNetCla,Fuhr06OrdEnr,LamStra05NamProCla,LamStr05ConsFreeBoo,Beletal06CatProThe,DJDHughes06PWS}
have only served to underline the difficulty of the problem.  Many of
these difficulties concern proofs with cuts. The identity of
non-analytic proofs is not problematic for intuitionistic logic; since
each proof has a unique normal form, he problem reduces to that of the
identity of normal proofs. Reduction to normal form in the classical
sequent calculus is in general neither confluent nor strongly
normalizing, and so the identity of proofs containing cuts must also
be considered.

Yet even for cut-free proofs, opinions on the ``right notion'' of
proof-identity differ. It is not reasonable, as it is for natural
deduction proofs, to declare two cut-free sequent proofs equal only if
they are syntactically identical; a good minimum notion of equality is
that proofs differing by \emph{commuting conversions} of
non-interfering sequent rules should be equal.
Proof-nets~\cite{Gir96ProNetPar} are a tool for providing canonical
representants of such equivalence classes of proofs in linear
logic~\cite{Girard87Linear}. A proposal by
Robinson~\cite{Rob03ProNetCla}, following ideas from
Girard~\cite{Girard91NewCon}, gives proof-nets for propositional
classical logic, and these nets do indeed identify proofs differing by
commutative conversions. However, they fail to provide canonical
representants for sequent proofs owing to the presence of
\emph{weakening attachments}; explicit information about the context
of a weakening not present in sequent proofs.  As a result one sequent
proof corresponds to many different nets, the exact opposite of the
situation one expects.  In addition, the proof-identities induced by
Robinson's nets do not include, among other desirable equations,
commutativity/associativity of contraction, a key assumption in the
development of abstract models of proofs (such identities are assumed
in \cite{Fuhr06OrdEnr}, in \cite{Beletal06CatProThe} and also in
\cite{LamStr05ConsFreeBoo}).  Other notions of abstract proof for
classical logic (Combinatorial proofs \cite{DJDHughes06TowHilbProb}
and $\mathbb{B}$/$\mathbb{N}$-nets \cite{LamStra05NamProCla}) make
such identifications, but at the cost of losing sequentialization into
a sequent calculus.

The current paper concerns \emph{expansion-nets}: a calculus of
proof-nets for classical logic first presented in~\cite{McK10exnets}
which, unlike Robinson's nets, provide canonical representants of
equivalence classes of classical sequent proofs. 
To avoid the problems inherent in weakening, we restrict attention to
proofs in a new sequent calculus, $\LKvar$ (see
Figure~\ref{fig:LKvar}).  This calculus has no weakening rule, nor
does it have implicit weakening at the axioms: instead, it has both
the multiplicative and additive forms of disjunction rule.  This new
calculus has all the properties one might hope of a sequent calculus
for classical logic (except, perhaps, terminating proof search): it
has the subformula property, is cut-free complete, and even has
syntactic cut-elimination (although this is perhaps easier to see via
the proof nets than directly in the sequent calculus, owing to the
curious nature of the cut-elimination theorem: if $\G$ is provable in
$\LKvar$ with cut, then some subsequent $\D \subseteq \G$ is provable
without cut). Treating the introduction of weak formulae in this way
allows us to define a canonical function mapping sequent proofs in $\LKvar$ to
expansion nets. 
 Correctness for
expansion nets (whether a net really corresponds to a sequent proof) can
be checked in polynomial time, using small adaptations of standard
methods from the theory of proof nets for \MLLminus + \mix\ (multiplicative linear
logic, plus the mix rule, without units, as studied in
\cite{Bel97SubMix,Dan90Thesis,LinLogPrimer}) -- meaning that
expansion-nets form a \emph{propositional proof
  system}~\cite{CookReckhow79PropProo}.  Translating from sequent
proofs to expansion nets identifies, in addition to nets differing by
commuting conversions, nets differing by the order in which
contractions are performed. The current paper (a self-contained extension of~\cite{McK10exnets})
 gives a detailed account of the connection between expansion-nets
and their associated sequent calculus: in particular, an
explicit proof of sequentialization for expansion
nets as(Theorem~\ref{thm:seq}), which was missing in~~\cite{McK10exnets}.
In addition, we present a cut-elimination procedure for expansion nets
(proof transformations which we prove, in Propositions
\ref{def:logcutandor} -- \ref{def:structcutdefweak} to preserve
correctness) which are weakly normalizing
(Lemma~\ref{lem:prinlem} and Theorem~\ref{thm:cut-elim} detail a
strategy for reducing any net with cuts to a cut-free net).
This result was absent from~\cite{McK10exnets}: with it, we can see that 
expansion nets have polynomial-time proof checking,
sequentialization into a sequent calculus and cut-elimination
preserving sequent-calculus correctness -- the first notion of abstract proof for
propositional classical logic to satisfy all of these properties.

\subsection{Structure of the paper}
\label{sec:structure-paper}
Section~\ref{sec:preliminaries} gives some preliminaries, and then
Section~\ref{sec:vari-sequ-calc} introduces the variant sequent
calculus $\LKvar$, showing completeness and some other key properties.
Section \ref{sec:exist-noti-proof} surveys the existing notions of
abstract proof in propositional classical logic. 
Section~\ref{sec:expansion-nets} defines expansion nets, and
compares them with the existing notions of abstract proof in the literature.

The next two chapters contain most of the novel technical material in
the paper.  Section~\ref{sec:subnets} deals with the notion of
\emph{subnet}, a key analogue of the notion of subproof in sequent
calculus which we will need to define cut-reduction.  This technology
(including the new notion of \emph{contiguous empire}) also affords a
proof of sequentialization of expansion-nets into $\LKvar$.
Section~\ref{sec:cut-elim-expans} then provides the cut-reduction
steps themselves, and a proof of cut-elimination for expansion nets.

\subsubsection{Acknowledgements}
\label{sec:acknowledgements}
The author thanks Kai Br\"unnler, Lutz Strassburger, Michel Parigot,
Tom Gundersen, and the anonymous referees for their helpful comments
and criticisms.

\begin{figure}
  \centering
  \noindent\hrulefill
  \[
  \begin{array}{cccc}

  & \begin{prooftree}
    \justifies  a, \ \bar{a}
    \using \Ax
  \end{prooftree} &\quad
  \begin{prooftree}
    \justifies  \top
    \using \Ax_\top
  \end{prooftree}
  \\[2em]
\begin{prooftree}
   \G, \ A \justifies \G, A \lor B
     \using \lor_{0}
\end{prooftree}
   \quad & \quad \begin{prooftree}
   \G, \ A,  B \justifies \G, A \lor B
     \using \lor
  \end{prooftree} \quad & \quad \begin{prooftree}
   \G, \ B \justifies \G,  A \lor B
     \using \lor_{1} \end{prooftree}
&  \quad \begin{prooftree}
    \G, A \qquad \D, B \justifies 
    \G, \D, \ A \land B \using \land
  \end{prooftree}
  \\[2em]
  \begin{prooftree}
    \G \qquad \D \justifies 
    \G, \D \using \mix
\end{prooftree}&\quad \begin{prooftree}
   \G,\ a,  a \justifies \G, a
    \using \C
  \end{prooftree}
& \quad
  \begin{prooftree}
   \G,\ \bar{a},  \bar{a} \justifies \G, \bar{a}
    \using \C
  \end{prooftree} & \quad\begin{prooftree}
   \G,\ A\land B,  A\land B \justifies \G, A\land B
    \using \C
  \end{prooftree}
  \\
  \end{array}
  \]
  \noindent\hrulefill
  \caption{$\LKvar$: A variant sequent calculus without weakening}
  \label{fig:LKvar}
\end{figure}

\section{Preliminaries}
\label{sec:preliminaries}
\subsection{Formulae of propositional classical logic}
\label{sec:form-line-class}

  Let $\mathcal{P}$ be a countable set of \emph{proposition symbols}. An
  \emph{atom} is a pair $(a, i)$, where $a \in \mathcal{P}$ and $i \in
  \finset{+, -}$.  By an abuse of notation, but in line with common
  use, we will simply write $a$ for $(a, +)$, and write $\bar{a}$ for
  $(a,-)$.  Two atoms are \emph{dual} if they differ only in their
  second component.



The classical formulae over $\mathcal{P}$ are given by the following
grammar

\[ A ::= a \ | \ \bar{a} \ | \ \top \ | \ \bot \ | \  A \land A  \ | \
A \lor A. \]

\newcommand{\rank}{\mathrm{rk}}
Negation is not a connective in our systems, but is defined by
De~Morgan duality.  We will use the notation $\bar{A}$ to denote the
De~Morgan dual of the formula $A$. The \emph{rank} $\rk(A)$ of a
formula $A$ is
defined as follows: 
\[ \rk (\top) = \rk(\bot)= \rk(a) = \rk(\bar{a}) = 1 \] 
\[ \rk(A\land
B) = \rk (A \lor B) = 1+ \mathrm{max}(\rk(A), \rk(B))\]

\newcommand{\pr}{\mathrm{pr}}

\subsection{Forests and sequents}
\label{sec:forests-sequents}

  A \emph{forest} (in this paper) is a pair $(A, \pr)$ consisting of a
  set $A$ of nodes and a partial endofunction $\pr$ (predecessor) on
  $A$ (the elements of $A$ on which $\pr$ is undefined being the
  \emph{roots}) such that, for each element $x$ of $A$, there is an $n
  \geq 0$ such that $\pr^n(x)$ is a root.  Clearly, a forest with one
  root is a tree.  Given a $y$ such that $\pr(x)=y$, we will say that
  $x$ is a \emph{successor} of $y$.  A node with no successors is a
  \emph{leaf}.  A node $x$ in a forest is \emph{ordered} if it comes
  equipped with an injective function from its set of successors to
  $\mathbb{N}$ --- otherwise it is \emph{unordered}. 

A forest defines a natural partial order $\leq$ on its nodes derived
from predecessor: $x \leq y$ if there exists $n \geq 0$ with $x =
\pr^n (y).$ A forest also gives rise to a directed graph (the graph of
the forest) with nodes the same as the nodes of the forest, and a
directed edge from every node to its predecessor.

A \emph{subforest} of $F$ is a nonempty set
  $G$ of nodes of $F$ such that if $g_1$ is a member of $G$
  and $g_1\leq g_2$ then $g_2$ is a member of $G$.


  Given that a formula is a tree, it is natural to consider a sequent
  to be a forest: a \emph{classical sequent} will be, for us, a finite
  forest whose trees are classical propositional formulae.

\begin{remark}
  Sequents are typically defined either as sets, multisets or
  sequences of formulae: why then have we chosen to define sequents as
  forests?  For an fine-grained analysis of proofs, sets are a bad
  representation, as they throw away all explicit information about
  contraction.  Sequences, on the other hand, distinguish too much;
  what we need is a representation which allows us to distinguish
  individual occurrences of the same formula in a sequent without
  caring in which order they appear.  The problem with the multiset
  representation of sequents lies in confusion over the meaning of
  ``multiset'', which is different depending on context, and in
  essential ways.  In particular, problems arise for structural proof
  theory if the intended meaning of multiset is ``set with
  multiplicities''.  Suppose that from $\bar{A}, A$ we
  derive $\bar{A}, A, A$ by weakening. If we wish to form a cut
  against $A$, we must choose which copy of $A$ to cut against: the
  choice will have drastic consequences during cut-elimination.  But
  in the ``set with multiplicities'' understanding of multisets, there
  is no notion of an individual copy of $A$ in the sequent. 

  By defining a sequent to be a forest, we avoid this conceptual
  hurdle: each formula in the sequent corresponds to a distinct root
  of the forest.  When we want to think about sequents as multisets to
  make sense, for example, of the expression $\Delta \subseteq \Gamma$
  (``$\Delta$ is a \emph{subsequent} of $\Gamma$\ ''), we can use the
  \emph{set} of roots of the sequent (the above expression is
  interpreted as ``$\Delta$ is a subforest of $\Gamma$, each of whose
  roots is a root of $\Gamma$\ '').
\end{remark}

We write sequent proofs without turnstiles: if $\mathbf{L}$ is a
sequent system, we write $\mathbf{L} \turnstile \G$ to mean ``there is
a sequent derivation in $\mathbf{L}$ with $\G$ at the root and axioms
at the leaves.

\begin{figure}
  \centering
  \noindent\hrulefill
  \[
  \begin{array}{ccc}
  \begin{prooftree}

    \justifies  a, \ \bar{a}
    \using \Ax
  \end{prooftree} &\quad
  & 
  \begin{prooftree}
    \justifies  \top
    \using \Ax_\top
  \end{prooftree}
  \\[2em]
   & \begin{prooftree}
   \G, \ A,  B \justifies \G, A \lor B
     \using \lor
  \end{prooftree}   \qquad \qquad
  \begin{prooftree}
    \G, A \qquad \D, B \justifies 
    \G, \D, \ A \land B \using \land
  \end{prooftree}
  \\[2em]
  \begin{prooftree}
   \G,\ A,  A \justifies \G, A
    \using \C
  \end{prooftree} &

&
  \begin{prooftree}
  \G \justifies \G, B \using \W
  \end{prooftree}
  \\
  \end{array}
  \]
  \noindent\hrulefill
  \caption{Cut-free multiplicative $\LK$ (one-sided)}
  \label{fig:LK}
\end{figure}
\section{A variant sequent calculus for classical logic}
\label{sec:vari-sequ-calc}
The completeness of expansion nets relies on the completeness of a
variant sequent calculus $\LKvar$ (shown in Figure~\ref{fig:LKvar}).
This sequent calculus was introduced, along with expansion-nets,
in~\cite{McK10exnets}.  The calculus bears some similarities to
Hughes's ``minimal calculus'' \Mp~\cite{DJDHughes:hybrid}, in that it
has both multiplicatively and additively formulated disjunction rules.
However, while \Mp\ has a mixed additive/multiplicative conjunction
rule, $\LKvar$ has the standard multiplicative conjunction rule. Given
these logical rules, we need the contraction rule (which is absent
from \Mp) to be complete with respect to classical logic.  This would
ordinarily make the multiplicative disjunction rule redundant, as it
is derivable from the two additive rules plus contraction; however, in
$\LKvar$ contraction is forbidden on disjunctions.  Contraction is,
however, admissible in $\LKvar$; we will prove this using the
following two easy lemmata:

\begin{lemma}[Pseudo-invertibility of $\lor$]
\label{lem:orinv}
  If $\LKvar \turnstile \G, A \lor B$, then one of the following holds:
  \begin{itemize}
    \item $\LKvar \turnstile \G, A, B$
    \item $\LKvar \turnstile \G, A$
    \item $\LKvar \turnstile \G, B$
  \end{itemize}
\end{lemma}

\begin{lemma}
\label{lem:topelim}
  If $\G$ is nonempty and $\LKvar \turnstile \G, \top$, then $\LKvar \turnstile \G$.
\end{lemma}


\begin{proposition}
Contraction is admissible in $\LKvar$.
\end{proposition}
\begin{proof}
  Contraction is admissible for $\top$ by Lemma~\ref{lem:topelim},
  and for atoms/conjunctions by the contraction rule.  Now suppose
  that contraction is admissible for all formulae of rank $< n$, and let
  $A\lor B$ have rank $n$.  Given a proof of $\G, A \lor B, A\lor B$,
  apply pseudo invertibility (Lemma~\ref{lem:orinv}) to obtain a proof
  of $\G, A^{(n)}, B^{(m)}$, (here $C^{(n)}$ denotes $n$ copies of
  the formula $C$) where $0 \leq m,n \leq 2$ and $n+m \geq 2$.
  Using a
  combination of the induction hypothesis and one of the disjunction
  rules of $\LKvar$ we obtain a proof of $\G, A\lor B$.\qed
\end{proof}
In common with \Mp, $\LKvar$ has the curious property of being sound
and complete for formulae ($\turnstile A$ iff $\vDash A$) but not
``sequent complete'': that is, there are sequents provable in $\LK$
which cannot be proved in the variant system.  For example, if $a$ and
$b$ are distinct propositional letters, then $a, \bar{a}, b$ does not
have a proof in $\LKvar$.  For this reason, our proof of completeness
proceeds by showing that each $\LK$-provable sequent has an
$\LKvar$-provable subsequent:

\begin{proposition}
  Let $\G$ be provable in $\LK$ (we take as $\LK$ the system in
  Figure~\ref{fig:LK}).  Then we may partition the formulae in $\G$
  (in terms of forests, the roots of $\G$) into $\G_s$ (the
  \emph{strong} formulae of $\G$) and $\G_w$ (the \emph{weak} formulae
  of $\G$), such that $\LKvar$ proves $\G_s$.
\end{proposition}
\begin{proof}
  By induction on the length of an $\LK$ derivation.  Clearly, the
  proposition is true for consequences of the $\LK$ axiom.  We proceed by
  case analysis on the last rule $\rho$ used in the $\LK$ derivation:
  \vspace{1em}

  \noindent \textbf{[$\rho=\W$]}\qquad The induction hypothesis gives
  us the strong formulae $\G_s$ of the premiss $\G$ of $\rho$, such
  that $\LKvar \turnstile \G_s$.  The sequent $\G_s$ is also a
  subsequent of the conclusion $\G, B$ of $\rho$, and so we may take
  it as the strong formulae of the conclusion (i.e. $B$ is a weak
  formula in the conclusion). \vspace{1em}

\noindent [$\rho = \lor$]\qquad  Let $\G, A,
B$ be the premiss of $\rho$, and $\G, A \lor B$ the conclusion.  Apply
the induction hypothesis to $\G, A, B$, yielding a sequent $\G_s$ of
strong formulae provable in $\LKvar$:
  \begin{itemize}
  \item If $A$ and $B$ are both strong, then $\G_s = \D, A, B$ is
    provable in $\LKvar$, and $\D, A\lor B$ is an $\LKvar$ provable
    subsequent of the conclusion of $\rho$.
  \item If $A$ and $B$ are both weak, then $\G_s$ is also a subsequent
    of the conclusion of $\rho$, and so we may take $\G_s$ as the
    strong formulae of the concluion of $\rho$.
  \item If $A$ is weak and $B$ is strong, then $\G_s = \D, B$, and
    thus, using $\lor_1$, $\D, A \lor B$ is an $\LKvar$ provable
    subsequent of the conclusion of $\rho$.  Symmetrically if $A$
    strong and $B$ weak.
  \end{itemize}
\vspace{1em}
\noindent [$\rho = \C$]\qquad  This is similar to the case for
disjunction, with the added twist that we must use admissible
contraction where a contraction rule is not available in $\LKvar$.
Let $\G, A, A$ be the premiss of $\rho$, and $\G, A$ the
conclusion.  Apply the induction hypothesis to $\G, A, A$, yielding a
sequent $\G_s$ of strong formulae provable in $\LKvar$:
  \begin{itemize}
  \item If both copies of $A$ are strong, then $\G_s = \D, A, A$ is
    provable in $\LKvar$, and $\D, A$ is an $\LKvar$ provable
    subsequent of the conclusion of $\rho$ by contraction
    admissibility.
  \item If both copies of $A$ are weak, then $\G_s$ is also a
    subsequent of the conclusion of $\rho$, and so we may take $\G_s$
    as the strong formulae of the conclusion of $\rho$.
  \item If one copy of $A$ is weak, then $\G_s = \D, A$ is also an $\LKvar$ provable
    subsequent of the conclusion of $\rho$. 
  \end{itemize}
\noindent[$\rho=\land$]\qquad This is the most interesting case.
Let $\G, A$, be one premiss of $\rho$ and $\D, B$ the other.  The induction hypothesis applied to both
premisses gives us a subsequents $\G_s$ and $\D_s$ of strong
formulae respectively for each premiss.
  \begin{itemize}
  \item If $A$ and $B$ are both strong in their respective sequents,
    then $\G_s = \G', A$ and $\D_s = \D', B$, and so $\G', \D', A\land
    B$, a subsequent of the conclusion of $\rho$, is provable in
    $\LKvar$.
   \item If $A$ and $B$ are both weak, then $\G_s, \D_s$ is a
     subsequent of the conclusion of $\rho$, provable in $\LKvar$
     using the $\mix$ rule. 
   \item If $A$ is weak and $B$ is strong, then $\G_s$ does not
     contain $A$, and is therefore a subsequent of $\G, \D, A \land B$
     provable in $\LKvar$. Symmetrically if $A$ strong and $B$ weak.
  \end{itemize}
\end{proof}

\begin{remark}
  $\LKvar$ is also formula complete without the $\mix$ rule; we only use
  $\mix\ $ in the completeness argument once, where a conjunction is
  applied to two weak formulae; the $\mix$ rule allows us to translate
  this derivation into $\LKvar$ in a symmetric manner.  Without
  $\mix$, we would be forced to choose one or other of the premisses
  as the strong formulae of the conclusion.
\end{remark}

\section{Existing notions of proof-net for classical logic}
\label{sec:exist-noti-proof}
To underline the need for a new notion of proof-net, we consider the
existing notions of proof-net for classical logic, and underline their
strengths and weaknesses as canonical representatives of equivalence
classes of proofs.  
\subsection{Na\"ive classical nets}
\label{sec:naive-classical-nets}
The basic idea for a rudimentary form of classical proof-net comes
from Girard~\cite{Girard91NewCon}, and the details were first worked
out by Robinson in \cite{Rob03ProNetCla}: the underlying structure of
the nets is identical to that for \MLL nets, and correctness is given
by treating the conjunctions and axioms of classical logic in the same
way as the linear logic axiom and tensor,treating both contraction and
disjunction in the same way as the linear logic ``par'' connective,
and treating weakenings as $\bot$ is treated in \MLL nets.

\begin{remark}
  The following presentation of classical nets differs from that of
  Robinson, in that we work with one-sided proofs, and we use
  weakening attachments for correctness rather than explicit weakening
  nodes.  Since these nets represent the most basic idea for
  developing \MLL nets into nets for classical logic, and since they
  lack many of the properties we would desire of proof-objects for
  classical logic, we call them \emph{na\"ive classical nets}.
\end{remark}
\begin{figure}[t]
  \noindent\hrulefill
  \center
  \begin{tikzpicture}[level distance=6.5mm, sibling distance=20mm, edge
    from parent path= {[<-](\tikzparentnode) to (\tikzchildnode)}]]
    \matrix[row sep= 0.5cm,column sep= 0.5cm]
    {
      
      \node {$\top$}[grow=up] child {node {$1$}};& &\node{$\Ax$}[grow
      = down, edge from parent path= {[->](\tikzparentnode) to
        (\tikzchildnode)}] child {node {$a$}} child {node {$\bar{a}$}};& &
      \node {$A$}[grow=up] child {node {Wk}};\\
      \node {$A \land B$}[grow=up] child {node (plus) {$\land$} child
        {node {$B$}} child {node {$A$}} };&& \node {$A$}[grow=up] child {node (plus) {Ctr} child {node {$A$}}
        child {node {$A$}} };  && \node {$A \lor B$}[grow=up]
      child {node (plus) {$\lor$} child
        {node {$B$}} child {node {$A$}} };\\

      \\
    };
  \end{tikzpicture}

  \noindent\hrulefill
  \caption{Na\"ive classical nets: graph figures}
  \label{fig:robnets}
\end{figure}
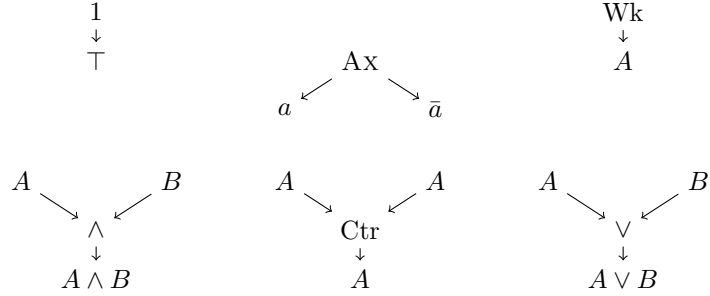

A graph-like presentation of na\"ive classical nets can be found in
Figure~\ref{fig:robnets}: a na\"ive classical proof-structure is a
graph built from the individual graph elements by matching types, such
that the resulting graph has no sources (nodes with no incoming edges)
labelled with formulae.  There is an inductive definition mapping
sequent-proofs in \LK to proof-structures, which can be very easily
obtained by considering proof-structures not as graphs, but as forests
of trees:

 \begin{definition}
  Let $\mathcal{X}$ be a countable set of \emph{wire symbols}. A
  \emph{wire variable} is an atom over $\mathcal{X}$, as defined in
  Section \ref{sec:form-line-class}: a pair of a member $x$ of
  $\mathcal{X}$ and a polarity ( $+$ or $-$).  Thus wire variables
  occur in dual pairs, for example $x$ and $\bar{x}$.  A
  \emph{contraction-weakening tree} (or \emph{cw-tree}) over
  $\mathcal{X}$ is a member of the following grammar.
  \[ s:: = 1 \ | \ \mathrm{Wk} \ | \   x \ | \ \bar{x} \ | \ (s \lor s) \ | \ (s
  \land s)  \ | \ \mathrm{Ctr}(s, s)\] 
  where $x$ and $\bar{x}$ are wire variables over $\mathcal{X}$.
\end{definition}
We use these cw-trees to define the mapping from sequent proofs to
proof structures, by \emph{annotating} formulae appearing in \LK
derivations with cw-trees. The system in Figure~\ref{fig:LKannot}
derives sequents of ``annotated formulae'', in which each formula has
an associated $cw$-tree: the tree attached to a formula provides a
history of how it was proved.
\begin{figure}
  \centering
  \noindent\hrulefill
  \[
  \begin{array}{ccc}
  \begin{prooftree}
    \justifies  x: a, \ \bar{x}: \bar{a}
    \using \Ax
  \end{prooftree} &
  \begin{prooftree}
    F \qquad G
    \justifies
    F, G
    \using \mix
  \end{prooftree}
  & 
  \begin{prooftree}
    \justifies  1:\top
    \using \Ax_\top
  \end{prooftree}
  \\[2em]
   & \begin{prooftree}
   G, \ t:A, \ s: B \justifies G, s \lor t : A \lor B
     \using \lor
  \end{prooftree}   \qquad
  \begin{prooftree}
    G,\ s: A \qquad F, \ t: B \justifies 
    G, F, \ s \land t: A \land B \using \land
  \end{prooftree}
  \\[2em]
  \begin{prooftree}
   G,\ s: A,\  t: A \justifies G, \ \mathrm{Ctr}(s,t): A
    \using \C
  \end{prooftree} &

&
  \begin{prooftree}
  G \justifies G, \ \mathrm{Wk}: B \using \W
  \end{prooftree}
  \\
  \end{array}
  \]
  \noindent\hrulefill
  \caption{$\LK_\mathrm{net}$: annotating $\LK$ plus $\mix$ with a
    na\"ive form of classical proof-net}
  \label{fig:LKannot}
\end{figure}

We can recover the more usual graph-like presentation of proof
structures by considering the \emph{graph} of an annotated sequent,
given by adding axiom links to the forest of cw-terms as suggested by
the dual wire variables.
\begin{example}
  The following annotated sequent represents a proof of Pierce's law
  \begin{equation}
((\var{\bar{x}} \lor \mathrm{Wk} ) \land \var{\bar{y}}):(\bar{p} \lor
q)\land \bar{p}, \quad  \mathrm{Ctr}(\var{x},
  \var{y}): p
\end{equation}
The graph of this annotated sequent is 
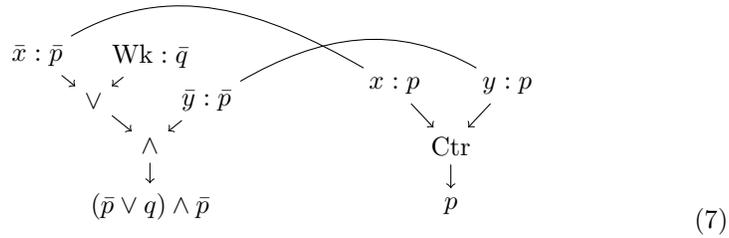
\begin{equation}
  \begin{tikzpicture}[level distance=8mm, grow=up, sibling
    distance=15mm, edge from parent path= {[<-](\tikzparentnode) to
      (\tikzchildnode)}]] 
    \node at (-2, 0) {$(\bar{p} \lor q)\land \bar{p}$} 
      child {node{$\land$} 
        child [level distance=6mm]{node (bary){$\bar{y}:\bar{p}$}}
        child [level distance=6mm]{node{$\lor$} 
          child [level distance=6mm]{node {$\mathrm{Wk}:\bar{q}$}}
          child [level distance=6mm]{node (barx){$\bar{x}:\bar{p}$}}}
      };
    \node at (2,0) {$p$} 
    child {node {$\mathrm{Ctr}$} 
      child {node (y) {$\var{y}:p$}} 
      child {node (x) {$\var{x}:p$}}
    };
    \draw (x) [in=30, out=150] to (barx);
    \draw (y) [in=30, out=150] to (bary);
\end{tikzpicture}\label{eq:1}
\end{equation}

\end{example}

To obtain a correctness criterion, it is necessary to
anchor each weakening to some other node of the proof.
In~\cite{Rob03ProNetCla} this anchoring is part of the structure of
the weakening node: we instead use the more usual notion of an
\emph{attachment}
\begin{definition}
  An \emph{attachment} $f$ for a na\"ive classical proof structure $F$
  is a function mapping each rule node labelled with Wk to some other
  rule-node of the proof-structure.  By an attached proof structure,
  we mean a pair $(F,f)$ of a proof structure $F$ and an attachment
  $f$ for $F$.
\end{definition}
\begin{example}
Below we see two different attachments of the same proof structure,
represented by the grey arrows:
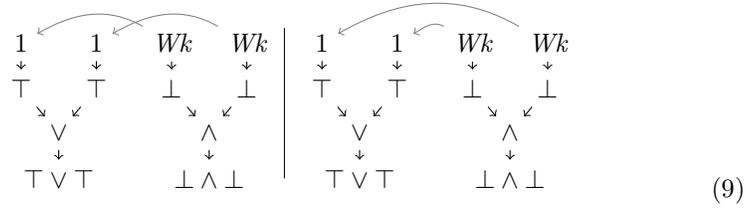
\begin{equation}
\label{eq:2}
  \begin{tikzpicture}[level distance=6mm, grow=up, sibling
    distance=10mm, edge from parent path= {[<-](\tikzparentnode) to
      (\tikzchildnode)}]] 
    \node at (-1, 0) {$\top \lor \top$} 
    child {node {$\lor$}  
        child {node {$\top$} child {node (X1) {$1$}}}
        child {node {$\top$} child {node (Y1) {$1$}}}};
\node at (1, 0) {$\bot \land \bot$} 
    child {node {$\land$}  
        child {node {$\bot$} child {node (X2) {Wk}}}
        child {node {$\bot$} child {node (Y2) {Wk}}}};
\draw [->, gray, in=30, out=150] (X2) to (X1) ;
\draw [->, gray, in=30, out=150] (Y2) to (Y1) ;

    \node at (3, 0) {$\top \lor \top$} 
    child {node {$\lor$}  
        child {node {$\top$} child {node (X1) {$1$}}}
        child {node {$\top$} child {node (Y1) {$1$}}}};
\node at (5, 0) {$\bot \land \bot$} 
    child {node {$\land$}  
        child {node {$\bot$} child {node (X2) {Wk}}}
        child {node {$\bot$} child {node (Y2) {Wk}}}};
\draw [->, gray, in=30, out=150] (X2) to (Y1) ;
\draw [->, gray, in=30, out=150] (Y2) to (X1) ;
\draw (2,2) to (2,0) ;
  \end{tikzpicture}
\end{equation}
\end{example}

The annotated sequent calculus in Figure~\ref{fig:LKannot} provides a
function from LK proofs to proof structures.  To extend this to
\emph{attached} proof structures, we must give an attachment for each
weakening in the sequent proof.  We may choose any one of the formulae
present in the context of the weakening rule; this arbitrary choice
means that attached proof-nets themselves cannot be the canonical
proof objects we seek.  For \MLL, the right notion of canonical proof
object is a \emph{quotient} of attached proof-nets by so-called
\emph{Trimble rewiring}~\cite{Trim94Thesis}, whereby two proof-nets
are equivalent if they can be transformed into one another by several
steps of ``rewiring'' a single unit: a rewiring is a change of
attachment for the unit which yields a correct net.  According to
Trimble rewiring, the two attached nets in \eqref{eq:2} are different,
as rewiring any one unit would result in a structure which is not a
net; this is important, as the corresponding morphisms are
distinguished in some $*$-autonomous categories.

The standard problem in the theory of proof-nets is to give a global
\emph{correctness criterion} for identifying, among the
proof-structures, those which can be obtained from desequentializing a
sequent proof.  This then leads to a \emph{sequentialization theorem},
allowing one to reconstruct a sequent proof out of a correct
proof-net. Na\"ive classical nets are very closely modelled on \MLL
nets; this means we may adapt any of the many equivalent formulations
of correctness for $\MLL$ nets to provide a correctness criterion for
them.  For example, the following is the switching graph criterion
\cite{DanReg89StrMul}, suitably altered for our setting:

\begin{definition} Let $F$ be a na\"ive classical proof-structure.

  \begin{enumerate}
  \item A rule-node of $F$ is \emph{switched}
    if it is a Ctr or $\lor$ node. A \emph{switching} of a na\"ive classical
    proof-structure is a choice, for each switched node, of one of its 
    successors.
  \item Given an attachment $f$ for $F$, and a switching $\sigma$ for
    $F$, the \emph{switching graph} $\sigma(F,f)$ is the graph
    obtained by deleting from $F$ all edges from a switched node to
    its successor not chosen by $\sigma$, forgetting directedness of
    edges, and adding an edge from each Wk node to its image
    under $f$.
  \item $(F,f)$ is \emph{ACC-correct} if, for
    each switching $\sigma$, $\sigma(F,f)$ is acyclic and connected. 
  \item $F$ is a \emph{na\"ive classical net} if, for some $f$, $(F,f)$ is ACC-correct.
  \end{enumerate}
\end{definition}

\begin{theorem}[Robinson]
  \begin{enumerate}
  \item   Every proof-structure arising from an $\LK$ proof is a
    na\"ive classical net.
  \item Every na\"ive classical net can be obtained by desequentializing an 
    $\LK$ proof.
  \end{enumerate}
\end{theorem}
\noindent Using the techniques developed in
\cite{Dan90Thesis,LinLogPrimer}, we can capture classical reasoning in
the presence of the $\mix$ rule (which does not allow us to prove any
new theorems, but extends the space of cut-free proofs):
\[
  \begin{prooftree}
    \turnstile \G \qquad \turnstile \D  \justifies \turnstile \G, \D
    \using \mix
  \end{prooftree}
\]

\begin{definition}  Let $F$ be a Robinson proof-structure, and $f$ an
  attachment for $F$
  \begin{enumerate}
  \item $(F,f)$ is \emph{AC-correct} if, for
    each switching $\sigma$, $\sigma(F,f)$ is acyclic. 
  \item $F$ is a \emph{\mix-net} if there is an
    attachment $f$ such that $(F,f)$ is AC-correct.
  \end{enumerate}
\end{definition}
\begin{theorem}
  \begin{enumerate}
  \item   Every proof-structure arising from a sequent proof in the
    system in Figure~\ref{fig:LK} plus $\mix$ is a $\mix$-net.
  \item Every \mix-net can be obtained by desequentializing a sequent
    proof with $\mix$.
  \end{enumerate}
\end{theorem}

Correctness for na\"ive classical nets, and sequentialization, can be
developed easily by analogy with \MLL nets; for details see \cite{Rob03ProNetCla}.

As intrinsic representations of proofs, na\"ive classical nets have a
number of drawbacks:

\subsubsection{Either correctness is \NP, or weakening introduces noncanonicity}
Correctness for na\"ive classical proof structures is \NP-complete; it
is in \NP, since the correctness criterion goes via guessing an
attachment for each $\mathrm{Wk}$: without an attachment it is not
possible to adapt the correctness criterion from \MLL.  Correctness
for unattached na\"ive nets is \NP\ hard since there is an evident
surjective map from cw-annotated sequents to unattached \MLL nets, for
which correctness is known to be \NP\
hard~\cite{LinWin92constant-only}.  We could, instead, take attached
na\"ive nets as our abstract proof objects, having, as in Robinson's
original formulation, an explicit attachent for each weakening.  Then
correctness would be checkable in polynomial time (so we would have a
propositional proof system) but there would no longer be a canonical
function mapping sequent proofs to proof-nets -- that is, we would not
have a calculus of abstract proofs.
\subsubsection{Contraction is not associative, commutative}
Given a $cw$-annotated sequent $F, \ t:A, \ s:A, \ u:A$, there are
twelve distinct ways to contract the three displayed terms in the
sequent calculus, each leading to a different na\"ive net.  For
example, the net 
\[ F, \  \mathrm{Ctr}(\mathrm{Ctr}(t, s),u):A \] 
is syntactically distinct from the net
\[ \ F, \ \mathrm{Ctr}(t,
\mathrm{Ctr}(s,u)):A\] 
\noindent  Na\"ive classical nets satisfy neither the identity
$\mathrm{Ctr}(\mathrm{Ctr}(t, s),u)=\mathrm{Ctr}(t,
\mathrm{Ctr}(s,u))$, nor the identity $\mathrm{Ctr}(s,t) =
\mathrm{Ctr}(t,s)$; taken together, these equations ensure a
canonical way to contract multiple instances of the same formula.

%
\subsubsection{Weakening is not a unit for contraction}
Given a net $G=F, t:A$, we can weaken to arrive at a net $F, t:A,
\mathrm{Wk}:A$, and then contract to form a net $F, \mathrm{Ctr}(t, \mathrm{Wk}):A$. This
net differs from $G$, but we would prefer it to be identified with
$G$: that is, $\mathrm{Wk}$ should be a unit for the contraction operation.

\subsubsection{Contraction on disjunctions is not pointwise}

Given a cw-annotated sequent \[F,\  t_1:A, \, t_2:A, \, s_1:B, \, s_2:B,\] we can
apply the rules of $\LK$ to obtain a single term of type $A\lor B$ in
five distinct ways, which once again we would prefer were identified.
Two of them are displayed below
\[F, \ \mathrm{Ctr}((t_1\lor s_1), (t_2\lor s_2)):A\lor B \ | \ F,\ 
(\mathrm{Ctr}(t_1, t_2)\lor \mathrm{Ctr}(s_1,s_2)):A \lor B. \] If these
two derivations, are identified, we will say that contraction on
disjunctions is \emph{constructed pointwise}: in na\"ive classical
nets this is clearly not the case.



Two further proposals for proof-net-like objects exist in the
literature. They do not suffer from the above problems but pay a heavy
price for doing so, lacking as they do  a strong connection with the
sequent calculus.  We will not discuss these proposals in as great a
depth as na\"ive classical nets, as there is not such a close
connection between them and expansion-nets.

\subsection{Lamarche-Strassburger nets}
\label{sec:lamarche-strassb-net}
The Lamarche-Strassburger approach to classical
proof-nets~\cite{LamStra05NamProCla} (hereafter LS-nets) are a
generalization of \MLLminus proof nets which allow classical logic to
be captured: instead of changing the underlying forests, as with
na\"ive proof structures, this approach changes the behaviour of the
\emph{links}.  Specifically, while in \MLLminus nets each leaf takes
part in precisely one axiom link, in LS-nets a leaf may take part in
several links, or indeed none -- it is this liberalized notion of
axiom link that allows LS-nets to capture classical logic.  Depending
on the particular flavour of net, there may even be more than one link
between a pair of dual atoms.  The ``proof-structures'' of these
calculi of nets are the following:

\begin{itemize}
\item A $\mathbb{B}$-prenet over $\G$ is a set $\mathcal{L}$ of pairs of
  leaves of $\G$, such that the first member of each pair is labelled
  with a positive atom $a$, and the second member of the pair is
  labelled with the dual $\bar{a}$ of that atom.
\item A $\mathbb{N}$-prenet over $\G$ is a multiset $\mathcal{L}$ of pairs of
  leaves of $\G$, such that the first member of each pair is labelled
  with a positive atom $a$, and the second member of the pair is
  labelled with the dual $\bar{a}$ of that atom.
\end{itemize}

\begin{figure}
  \centering
  \begin{tikzpicture}[level distance=6.5mm, grow=up, sibling
    distance=7mm, edge from parent path= {[<-](\tikzparentnode) to
      (\tikzchildnode)}]]
\node at (-2.8,0) {$\land$} child {node
  (a1) {$a$}} child {node (barb1) {$\bar{b}$}};
\node at (-1.4,0) {$\land$} child {node
  (barb2) {$\bar b$}} child {node (bara1) {$\bar{a}$}} ;
\node at (0,0) {$\land$}  child {node
  (a2) {$a$}} child {node (b1) {$b$}};
\node at (1.4,0) {$\land$}child {node
  (b2) {$b$}} child {node (bara2) {$\bar{a}$}};
\draw [color=gray,out = 60, in = 120] (barb1) to (b2);
\draw [color=gray,out = 50, in = 130] (barb1) to (b1);
\draw [color=gray,out = 50, in = 130] (barb2) to (b2);
\draw [color=gray,out = 30, in = 150] (barb2) to (b1);
\draw [color=gray,out = 30, in = 150] (a1) to (bara1);
\draw [color=gray,out = 30, in = 150] (a2) to (bara2);
  \end{tikzpicture} \quad   \begin{tikzpicture}[level distance=6.5mm, grow=up, sibling
    distance=7mm, edge from parent path= {[<-](\tikzparentnode) to
      (\tikzchildnode)}]]
\node at (-2.8,0) {$\land$} child {node
  (a1) {$a$}} child {node (barb1) {$\bar{b}$}};
\node at (-1.4,0) {$\land$} child {node
  (barb2) {$\bar b$}} child {node (bara1) {$\bar{a}$}} ;
\node at (0,0) {$\land$}  child {node
  (a2) {$a$}} child {node (b1) {$b$}};
\node at (1.4,0) {$\land$}child {node
  (b2) {$b$}} child {node (bara2) {$\bar{a}$}};
\draw [color=gray,out = 60, in = 120] (barb1) to (b2);
\draw [color=gray,out = 50, in = 130] (barb1) to (b1);
\draw [color=gray,out = 50, in = 130] (barb2) to (b2);
\draw [color=gray,out = 30, in = 150] (barb2) to (b1);
\draw [color=gray,out = 30, in = 150] (a1) to (bara1);
\draw [color=gray,out = 60, in = 120] (a1) to (bara1);
\draw [color=gray,out = 30, in = 150] (a2) to (bara2);
  \end{tikzpicture}\quad 

 \begin{tikzpicture}[level distance=7.5mm, grow=up, sibling
    distance=7mm, edge from parent path= {[<-](\tikzparentnode) to
      (\tikzchildnode)}]]
    \node at (-7,0) (c1l){$\land$} child
    {node (a1) {$a$}} child {node (barb1){$\bar{b}$}};
\node at (-5.6,0) (c2l){$\land$} child {node
  (barb2) {$\bar b$}} child {node (bara1) {$\bar{a}$}};
\node at (-4.2,0) (c1r){$\land$} child {node
  (a2) {$a$}} child {node (barb3) {$\bar{b}$}};
\node at (-2.8,0) (c2r){$\land$} child {node
  (barb4) {$\bar b$}} child {node (bara2) {$\bar{a}$}};
\node at (-1,-0.75) {$\land$} child {node
  (a3) {$a$}} child {node (b1) {Ctr} child {node (b12) {$b$}}child {node (b11) {$b$}}};
\node at (0.5,-0.75) {$\land$}  child {node
  (b2) {Ctr}child {node (b22) {$b$}}child {node (b21) {$b$}}} child {node (bara3) {$\bar{a}$}};
\node at (-5.6,-0.75) (ctr1) {Ctr};
\draw[->] (c1l) to (ctr1);
\draw[->] (c1r) to (ctr1);
\node at (-4.2,-0.75) (ctr2) {Ctr};
\draw[->] (c2l) to (ctr2);
\draw[->] (c2r) to (ctr2);
\draw[color=gray, out = 60, in = 120] (barb1) to (b22);
\draw[color=gray, out = 60, in = 120] (a1) to (bara1);
\draw[color=gray, out = 60, in = 120] (a2) to (bara2);
\draw[color=gray, out = 60, in = 120] (a3) to (bara3);
\draw[color=gray, out = 60, in = 120] (barb2) to (b21);
\draw[color=gray, out = 60, in = 120] (barb3) to (b12);
\draw[color=gray, out = 80, in = 100] (barb4) to (b11);
  \end{tikzpicture}
  \caption{The same $\LK^*$ proof, rendered as a $\mathbb{B}$-net, $\mathbb{N}$-net, and na\"ive net}
  \label{fig:BvsNvsnaive}
\end{figure}
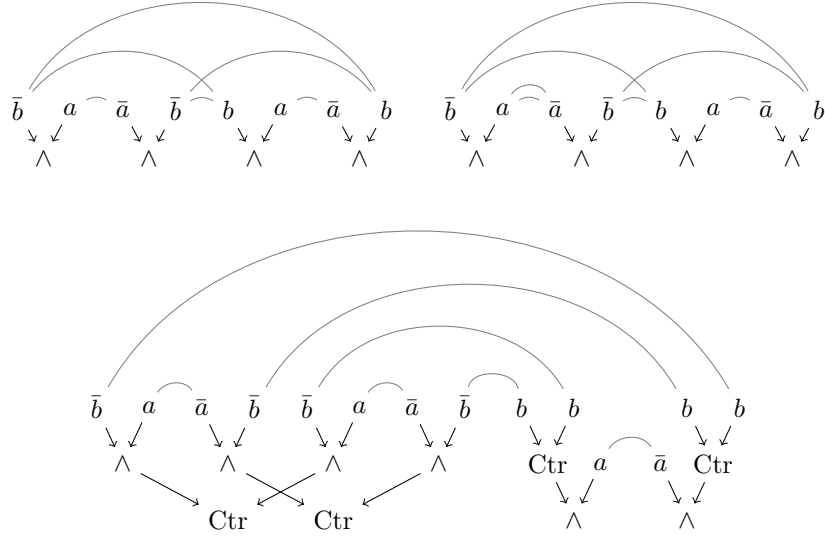

The difference between $\mathbb{B}$-prenets, $\mathbb{N}$-prenets and
na\"ive classical nets can be readily seen in
Figure~\ref{fig:BvsNvsnaive}: while contraction is explicit in na\"ive
nets, in a $\mathbb{B}$-prenet it is represented by an atom's
participation in multiple axiom links. In a $\mathbb{N}$-prenet, there
can, in addition, be multiple links between the same pair of atoms:
thus more information about contraction is present in naive nets than
in $\mathbb{N}$-nets, and more in $\mathbb{N}$-nets than in
$\mathbb{B}$-nets.

The translation from sequent proofs to pre-nets is almost immediate:
it arises simply by tracing the occurrences of atoms through the
sequent proof (for full details see \cite{LamStra05NamProCla}).  If we
are interested in extracting a $\mathbb{B}$-prenet, we only care if
there is a path between two atoms: in the case of $\mathbb{N}$-prenets
we are also interested in \emph{how many} paths there are.  Neither
flavour of LS-net suffers from the non-canonicity problems of
Robinson-style nets, but they introduce new problems:

\subsubsection{No polynomial-time correctness algorithm for $\mathbb{B}$-nets}
Strassburger and Lamarche give in \cite{LamStra05NamProCla} an exponential-time
criterion singling out those $\mathbb{B}$-prenets which correspond to
sequent proofs; since the size of a $\mathbb{B}$-net is polynomially
bounded by the size of its conclusion, we cannot reasonably hope to do
better. The condition given for $\mathbb{N}$-nets in
\cite{LamStra05NamProCla} simply collapses a $\mathbb{N}$-net to a
$\mathbb{B}$-net and checks correctness of the $\mathbb{B}$-net, the
result being that there are ``correct'' $\mathbb{N}$-nets that are not
the translation of any sequent proof. There is some hope that a
different polynomial-time correctness criterion might be found for
these nets, or for the similar \emph{atomic
  flows}~\cite{GugGunStra10BreFlows}, but none has been found so far,
despite substantial effort.  Consequently, there is currently no
notion of sequentializing $\mathbb{N}$-nets, either into a sequent
system or some other calculus.

\subsubsection{Cut-elimination does not preserve  correctness}
Cut-elimination is easy to define on LS-nets: as shown
in~\cite{LamStra05NamProCla}, it suffices, when opposing atomic
contractions in a cut, to simply count the number of paths through the
cut between each pair of atoms. This procedure is proved
in~\cite{LamStra05NamProCla} to be strongly normalizing, confluent,
and correctness preserving on $\mathbb{B}$-nets. However, applying
this procedure to $\mathbb{N}$-nets, there is a $\mathbb{N}$-net which is
the image of a sequent proof, but whose cut reduct is not the image of
a sequent proof; cut-reduction does not preserve correctness with
respect to the sequent-calculus.

\subsection{Hughes's Combinatorial proofs}
\label{sec:hugh-comb-proofs-1}
The combinatorial proofs of
Hughes~\cite{DJDHughes06PWS,DJDHughes06TowHilbProb} are a more radical
departure from the standard notions of proof net than
Lamarche-Strassburger or Robinson-style nets.  Broadly, combinatorial
proofs represent classical proofs as ``fibered'' linear proofs, with
the fibring representing the structural rules.  The
``semi-combinatorial'' presentation of combinatorial proofs given in
\cite{DJDHughes06TowHilbProb} is the most immediately graspable for a
proof-theorist: a combinatorial proof of a sequent $\G$ of classical
propositional logic is a function $f$ from the leaves of an $\MLLmix$
proof net $\pi$ (which we can represent as a \emph{binary} \MLL
formula, in which atoms occur in dual pairs) to the leaves of $\G$
preserving
\begin{itemize}
\item Duality (if leaves $X$ and $Y$ are dual, then so are $f(X)$ and
  $f(Y))$
\item Conjunctive relationships (If the topmost connective between $X$
  and $Y$ is a $\ox$, then the topmost connective between $f(X)$ and
  $f(Y)$ is a $\land$.
\end{itemize}
and such that $f$ is a \emph{contraction-weakening}: 
\begin{itemize}
\item $f$ is built from pure contraction ($c: A \land A \to A$), weakening
($w: A\land B \to A$), and associativity/commutativity of the
connectives, using function composition and ``horizontal'' composition
(if $f: A \to B$ and $g: A' \to B'$ are contraction-weakenings, then
so are the evident functions $f\land g: A \land A' \to B \land B'$ and
$f\lor g: A \lor A' \to B \lor B'$).
\end{itemize}
\begin{example}
\label{ex:combproof}
  An example of a semi-combinatorial proof is the following:
  
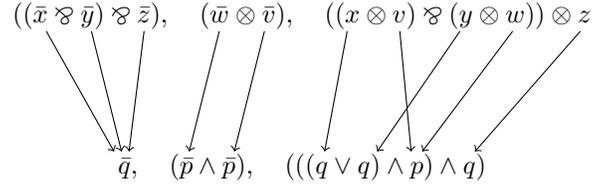
\[
  \begin{tikzpicture}
    \node (top) at (0,2){ $((\bar{x} \Par \bar{y}) \Par \bar{z}),
      \quad (\bar{w} \ox \bar{v}), \quad ((x \ox v)\Par
      (y\ox w)) \ox z$}; \node (bottom) at (0,0) {$\bar{q}, \quad
      (\bar{p}\land \bar{p}), \quad (((q \lor q) \land p)\land q)$};
    \draw [->]($(top)+(-3.4,-.2)$) to ($(bottom)+(-2.5,.2)$); \draw
    [->]($(top)+(-2.8,-.2)$) to ($(bottom)+(-2.4,.2)$); \draw
    [->]($(top)+(-2.1,-.2)$) to ($(bottom)+(-2.3,.2)$); \draw
    [->]($(top)+(-1.1,-.2)$) to ($(bottom)+(-1.5,.2)$); \draw
    [->]($(top)+(-0.5,-.2)$) to ($(bottom)+(-0.9,.2)$); \draw
    [->]($(top)+(0.6,-.2)$) to ($(bottom)+(0.3,.2)$); \draw
    [->]($(top)+(1.3,-.2)$) to ($(bottom)+(1.45,.2)$); \draw
    [->]($(top)+(2.1,-.2)$) to ($(bottom)+(1,.2)$); \draw
    [->]($(top)+(2.8,-.2)$) to ($(bottom)+(1.6,.2)$); \draw
    [->]($(top)+(3.7,-.2)$) to ($(bottom)+(2.3,.2)$);
  \end{tikzpicture}
  \]
\end{example}
Semi-combinatorial proofs suffer from the same problems as na\"ive
nets with regard to associativity of contraction: differences in the
association of contractions manifest in the $\MLLmix$ formula: for
example, the following is also a semi-combinatorial proof, differing
from the one above only by the association of the left-hand $\Par$:

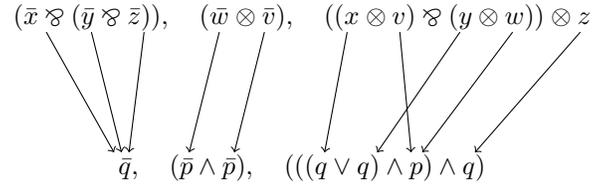
\[ 
\begin{tikzpicture}
  \node (top) at (0,2){ $(\bar{x} \Par  (\bar{y} \Par \bar{z})), \quad (\bar{w} \ox
  \bar{v}), \quad ((x \ox v)\Par
  (y\ox w)) \ox z$};
\node (bottom) at (0,0)  {$\bar{q}, \quad (\bar{p}\land \bar{p}), \quad (((q
  \lor q) \land p)\land q)$};
\draw [->]($(top)+(-3.4,-.2)$) to ($(bottom)+(-2.5,.2)$);
\draw [->]($(top)+(-2.8,-.2)$) to ($(bottom)+(-2.4,.2)$);
\draw [->]($(top)+(-2.1,-.2)$) to ($(bottom)+(-2.3,.2)$);
\draw [->]($(top)+(-1.1,-.2)$) to ($(bottom)+(-1.5,.2)$);
\draw [->]($(top)+(-0.5,-.2)$) to ($(bottom)+(-0.9,.2)$);
\draw [->]($(top)+(0.6,-.2)$) to ($(bottom)+(0.3,.2)$);
\draw [->]($(top)+(1.3,-.2)$) to ($(bottom)+(1.45,.2)$);
\draw [->]($(top)+(2.1,-.2)$) to ($(bottom)+(1,.2)$);
\draw [->]($(top)+(2.8,-.2)$) to ($(bottom)+(1.6,.2)$);
\draw [->]($(top)+(3.7,-.2)$) to ($(bottom)+(2.3,.2)$);
\end{tikzpicture}
\]
Combinatorial proofs themselves avoid this problem by
representing the binary $\MLLmix$ theorem not as a formula, but as its
\emph{co-graph}: two \MLL formulae have the same co-graph
if and only if they differ by associativity and commutativity of
connectives.  Thus, combinatorial proofs provide a sufficiently
abstract notion of proof for our purposes.

The contraction-weakening requirement is equivalent to two other
requirements, as proved by Hughes: the \emph{skew fibration} condition
and the fact that $f$ preserves maximal cliques of conjunctively
related leaves.  The surprising result
of~\cite{DJDHughes06TowHilbProb} is that these conditions can be
checked in polynomial time: thus Combinatorial proofs, unlike
unattached na\"ive classical nets or LS-nets, form a propositional
proof system.


Combinatorial proofs fail to satisfy our other two
specifications for a good notion of abstract classical proof:

\subsubsection{Sequentialization into a nonstandard calculus}
There are combinatorial proofs which are not the image of any
sequent-calculus proof, as shown in~\cite{DJDHughes06TowHilbProb};
Hughes introduces in that paper an extended calculus (the
\emph{Homomorphism calculus} for which the map from proofs to
invariants is surjective.  This calculus can be seen as a
generalization of the sequent calculus which replaces the usual 
structural rules with a homomorphism rule
\[
\begin{prooftree}
  \G, A  
\justifies \G, B \using \mbox{$f:A \to B$ is a contraction-weakening}
\end{prooftree}
\]
 but is less well understood
than the sequent calculus: in addition, it lacks certain desirable
properties, such as the subformula property.
\subsubsection{Cut-reduction does not preserve sequent correctness}
We might hope that some other, more sophisticated correctness
condition might identify the combinatorial proofs arising from sequent
calculus derivations.  This may be so, but such a correctness
criterion would be incompatible with the dynamic aspects of
combinatorial proofs shown in~\cite{DJDHughes06TowHilbProb}.  In that
paper Hughes defines a notion of combinatorial proof with cut, gives a
strongly normalizing cut-elimination procedure for combinatorial
proofs which preserves his correctness criterion.  However, this
procedure does not stay within this subclass of sequent-correct
combinatorial proofs. 
\label{sec:hugh-comb-proofs}

\section{Expansion nets}
\label{sec:polyt-check-class}
As we saw in the previous section, weakening causes substantial
problems in na\"ive classical proof-nets, but the alternatives
($\mathbb{N}$-nets and combinatorial proofs) lack
correctness/sequentialization with respect to a sequent calculus.  In
this section we give a calculus of nets which retains a connection to
the sequent calculus while also having a polynomial-time correctness
criterion, without the need for weakening attachment and its attendant
noncanonicity.

The basic idea can be seen already in na\"ive classical nets: if
weakening only happens within a disjunction, then attachment is
redundant.  Let $F$ be a na\"ive proof-structure.  If a weakening
subterm $\mathrm{Wk}$ of $F$ is the successor of a disjunction, and if
the other successor $t$ of that disjunction is not an instance of
$\mathrm{Wk}$, we will say that the weakening subterm has a
\emph{default attachment}, namely $t$.  If every weakening subterm of
$F$ has a default attachment, we will say it is
\emph{default-attached}.  If $F$ is default-attached, the
\emph{default attachment} of $F$ is the function from instances of
$\mathrm{Wk}$ to nodes of $F$ assigning each instance of $\mathrm{Wk}$
to its default attachment.

\begin{example}
  The net \eqref{eq:2} for Pierce's formula is default attached: the
  only weakening in that net appears as an immediate subtree of a
  disjunction, and the setting the other disjunct $\bar{x}:\bar{p}$ as the
  attachment for it yields an ACC correct attached net.

  The cw-annotated sequent $x:p, \ \bar{x}: \bar{p}, \ \mathrm{Wk}:q$ is not
  default-attached, as the weakening appears outside of a disjunction.
  The following is also not default-attached: 
\[  1: \top, \ \mathrm{Wk} \land \mathrm{Wk}:\bot \land \bot, \ 1:\top, \ \mathrm{Wk}
\land \mathrm{Wk}:\bot \land \bot,\  1:\top \] 
\end{example}

Since the difficult part of correctness for na\"ive nets is guessing
the attachment, correctness for default-attached nets is easy:

\begin{proposition}
  Correctness of default-attached na\"ive proof-structures can be
  checked in polynomial time.
\end{proposition}
\begin{proof}
  Correctness for na\"ive structures is \NP\ because the attachment of
  the weakenings must be guessed.  For a default-attached structure,
  the default attachment can be computed in linear time, and the
  polynomial correctness algorithm for attached nets may then be
  applied. \qed
\end{proof}

\label{sec:expansion-nets}
Default-attached nets improve on general na\"ive nets by having a
polynomial-time verifiable correctness criterion, without the need for
an explicit weakening attachment (which compromises the canonicity of
na\"ive nets). However, we still have the problem that contraction is
neither associative, commutative, nor pointwise on disjunctions. The
first two of these problems were noticed by Girard at the same time he
proposed nets for classical logic, and there is an evident solution:
make contraction n-ary, while at the same time forbidding either
weakenings or contractions from being the successors of a contraction.
The last of these problems (pointwise contraction) can be solved by
forbidding contraction on disjunctions.  We enforce those conditions
by moving to a new kind of proof-net, which we call \emph{expansion
  nets}: these nets were introduced in \cite{McK10exnets}. The
terminology is inspired by Miller's \emph{expansion-tree
  proofs}~\cite{Mil87ComRep}, which are a representation of proofs in
first- and higher-order logic.  Expansion-tree proofs represent n-ary
contraction in a similar fashion to expansion nets; in expansion trees
contraction happens only on existentially quantified subformulae (not
on universally quantified formula), and is represented by formal sums
(expansions) of witnessing terms rather than binary
contractions. Expansion-tree proofs provide a compact,
bureaucracy-free representation of proofs for first- and higher-order
classical logic; expansion-nets provide a similar technology for
propositional classical logic.

Expansion-nets are built from trees we call \emph{propositional
  expansion trees} (to distinguish from Miller's expansion trees):

\begin{definition}[Propositional Expansion trees]
  \label{def:alel-terms}
  Let $\mathcal{X}$ be a set of wire symbols, with = $\var{x},
  \var{y}, \bar{x}, \bar{y}\dots$ the corresponding \emph{wire
    variables} -- atoms over $\mathcal{X}$.  An \emph{propositional
    expansion tree} over $\mathcal{X}$ is of the form $t$ below:
  \[ t ::= 1 \ | \ ( w + \dots + w) \ | \ (t\lor t) \ | \ (t \lor *) \
  | \ (* \lor t) \qquad w ::= \var{x} \ | \ \var{\bar{x}} \ | \
  t\otimes t \]
  \noindent where $(w +\dots + w)$ denotes a nonempty finite formal
  sum, which we call an \emph{expansion}.  We call the members of the
  grammar $w$ ``witnesses''.  In line with the previous section we
  will call trees of the form $(t \lor *)$ and $(* \lor t)$
  \emph{default weakenings}.
\end{definition}

Just as cw-trees gave us a succinct way to write down and reason about
na\"ive nets, so propositional expansion trees will give us a nice way
to present expansion nets.  However, it will be just as important to
think of expansion-nets as a graphical proof calculus, in particular
when we want to talk about paths in a net.  For this purpose, we will
need to consider the tree (in the sense of
Section~\ref{sec:forests-sequents}) defined by a propositional
expansion tree: that is, a set of nodes and a predecessor function.
We should also consider which of the nodes of this tree are ordered.

The parse tree for an expansion-tree/witness (given by the grammars in
Definition~\ref{def:alel-terms}) gives us an immediate reading of a
propositional expansion tree (or witness) as a tree: for example, the propositional
expansion trees
\[   (\bar{x} + \bar{y} + \bar{z}) \quad ((\bar{w}) \ox (\bar{v})) \quad ((((* \lor (x)) \ox (v))+ (((y) \lor *)\ox (w))) \ox (z)) \] 
can be seen as trees
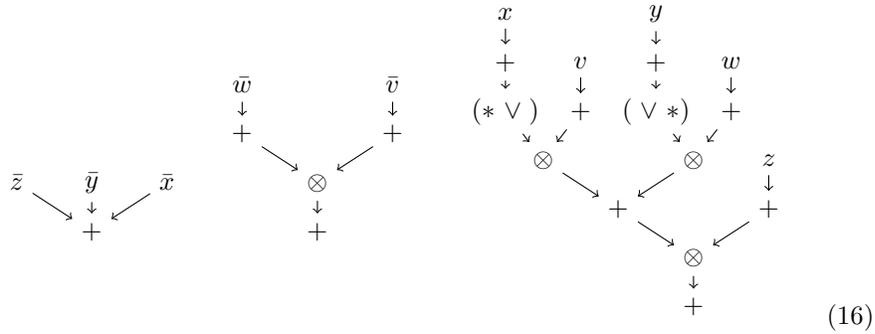
\begin{equation}
  \begin{tikzpicture}[level distance=6.5mm, grow=up, sibling
    distance=20mm, edge from parent path= {[<-](\tikzparentnode) to
      (\tikzchildnode)}]]
  \node at (-3, 1) {$+$} 
    child [sibling
    distance=10mm]{node {$\bar{x}$}} 
    child [sibling
    distance=10mm] {node{$\bar{y}$}}
    child [sibling
    distance=10mm] {node{$\bar{z}$}};
  \node at (0, 1) {$+$} 
    child {node {$\ox$} 
      child {node{$+$} child {node{$\bar{v}$}}}
      child {node{$+$} child {node{$\bar{w}$}}}};
  \node at (5, 0) {$+$} 
    child {node {$\ox$} 
      child {node{$+$} child {node{$z$}}}
      child {node{$+$} child [sibling
    distance=20mm] {node{$\ox$} 
                               child [sibling
    distance=10mm]{node{$+$} child {node{$w$}}}
                               child [sibling
    distance=10mm] {node {$(\;  \lor \ *)$} child
                                 {node{$+$} child {node{$y$}}}}}
                       child [sibling
    distance=20mm]{node{$\ox$} 
                               child [sibling
    distance=10mm] {node{$+$} child {node{$v$}}}
                               child [sibling
    distance=10mm] {node {$(* \ \lor \; ) $} child {node{$+$} child {node{$x$}}}}}
}};
\end{tikzpicture}
\end{equation}
However, this tree-reading of an expansion-tree treats the subtrees
$(t\lor *)$ and $(* \lor t)$ as having only one successor.  It will be
useful at certain points to regard $*$ as a subtree of $(t\lor *)$
(resp $(* \lor t)$) even though the symbol $*$ never appears outside
of a default weakening.  Treating the occurrences of $*$ as nodes, we
obtain the tree
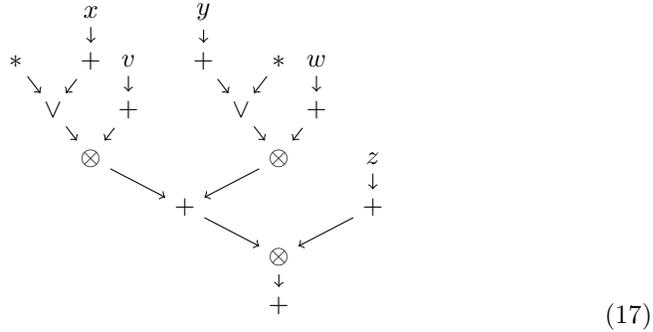
\begin{equation}
  \begin{tikzpicture}[level distance=6.5mm, grow=up, sibling
    distance=25mm, edge from parent path= {[<-](\tikzparentnode) to
      (\tikzchildnode)}]]
  \node at (0, 0) {$+$} 
    child {node {$\ox$} 
      child {node{$+$} child {node{$z$}}}
      child {node{$+$} child [sibling
    distance=25mm] {node{$\ox$} 
                               child [sibling
    distance=10mm]{node{$+$} child {node{$w$}}}
                               child [sibling
    distance=10mm] {node {$\lor$} child {node{$*$}} child
                                 {node{$+$} child {node{$y$}}}}}
                       child [sibling
    distance=25mm]{node{$\ox$} 
                               child [sibling
    distance=10mm] {node{$+$} child {node{$v$}}}
                               child [sibling
    distance=10mm] {node {$\lor$} child {node{$+$} child {node{$x$}}}
                                  child {node{$*$}}}}
}};
\end{tikzpicture}
\end{equation}

  We will call the nodes of a propositional expansion tree
which are not instances of $*$ \emph{proper nodes}.

When showing examples of expansion-nets, we will sometimes 
not show the expansion structure on trivial expansions of atomic type:
this improves readability and makes some diagrams smaller.  For
example, using this shorthand the three expansion trees above are:
\begin{equation}
  \begin{tikzpicture}[level distance=6.5mm, grow=up, sibling
    distance=20mm, edge from parent path= {[<-](\tikzparentnode) to
      (\tikzchildnode)}]]
  \node at (-3, 1) {$+$} 
    child [sibling
    distance=10mm]{node {$\bar{x}$}} 
    child [sibling
    distance=10mm] {node{$\bar{y}$}}
    child [sibling
    distance=10mm] {node{$\bar{z}$}};
  \node at (0, 1) {$+$} 
    child {node {$\ox$} 
      child {node{$(\bar{v})$}}
      child {node{$(\bar{w})$}}};
    \node at (5,0) {$+$} 
    child {node {$\ox$} 
      child {node{$(z)$}}
      child {node{$+$} child [sibling
    distance=25mm] {node{$\ox$} 
                               child [sibling
    distance=10mm]{node{$(w)$}}
                               child [sibling
    distance=10mm] {node {$\lor$} child {node{$*$}} child
                                 {node{$(y)$}}}}
                       child [sibling
    distance=25mm]{node{$\ox$} 
                               child [sibling
    distance=10mm] {node{$(v)$}}
                               child [sibling
    distance=10mm] {node {$\lor$} child {node{$(x)$}}
                                  child {node{$*$}}}}
}};
\end{tikzpicture}
\end{equation}

The successors of a node $t \lor s$ or $t \ox s$ are the nodes $t$ and
$s$: these nodes are ordered, as they correspond to the
sequent-calculus introduction rules for the connectives.  The
successors of a node $(w_1+ \cdots + w_n)$ are the nodes $w_1$ to
$w_n$. Since $+$ denotes a formal sum, the successors of an expansion
are \emph{unordered}: this corresponds to the fact that contraction is
a symmetric operation.


The successors of a node $(t \lor *)$ are the node $t$ and a node
labelled $*$, ordered such that the order of $t$ is 0 and the order of
the $*$ is 1.  Similarly for $(* \lor t)$, but with the orders
reversed.  The nodes labelled with $*$, $x$, $\bar{x}$ and $1$ have no
successors: they are the leaves of the tree.

Our proof structures will be \emph{typed} forests of propositional
expansion trees: we type propositional expansion trees with formulae
of classical propositional logic.  To maintain the associativity and
commutativity of the formal sum (which interprets contraction), we
make a distinction at the level of types between witnesses and
expansions: the expansions recieve a special ``witness types'', while
the expansion is typed with a formula.  This enforces that
contractions are n-ary and of maximum size.

\begin{definition}
  \label{def:types}
  A \emph{type} is either
  \begin{enumerate}
  \item A formula of classical propositional logic;
  \item A \emph{witness type}: one of the three following forms:
    \begin{itemize}
    \item A \emph{positive witness type}, written $[a]$, where $a$ is
      a positive atom;
    \item A \emph{negative witness type}, written $[\bar a]$, where
      $\bar{a}$ is a negative atom; or
    \item A conjunctive witness type, written $A \otimes B$, where $A$
      and $B$ are formulae of propositional classical logic.
    \end{itemize}
  \end{enumerate}
  Each witness type has an underlying classical formula: for $A\ox B$
  this is $A \land B$, for $[a]$ this is $a$ and for $[\bar{a}]$ this
  is $\bar{a}$.

  A \emph{typed tree}/\emph{typed witness} is a pair of a
  propositional expansion tree/witness and a type, derivable in the
  typing system shown in Figure~\ref{fig:typing-derivations}. This
  typing system should be thought of as an analogue of
  Figure~\ref{fig:robnets} for expansion-nets: it specifies the shape
  of the ``proof-structures'' we consider.
\end{definition}


\begin{example}
 The wire variable $x$ can be assigned the witness type $[p]$, while the expansions
$(x)$ (a \emph{trivial} expansion) and $(x + y)$ can be assigned as a type the
propositional formula $p$.
\end{example}

\begin{example}
\label{ex:forestexample}
The following are correctly typed propositional expansion trees:
  \[ (\bar{x} + \bar{y} + \bar{z}):\bar{q} \qquad ((\bar{w}) \ox
  (\bar{v})):(\bar{p}\land \bar{p})  \]\[((((* \lor (x)) \ox (v))+
  (((y) \lor *)\ox (w))) \ox (z)):(((q \lor q) \land p)\land q)  
\]
\end{example}

\begin{figure}[t]
  \noindent\hrulefill
  \centering
  \[
  \begin{prooftree}
    \justifies \var{\bar{x}} : [\bar{p}]
  \end{prooftree} \qquad
  \begin{prooftree}
    \justifies 1 : \top
  \end{prooftree}
  \qquad \begin{prooftree} t:B \justifies (*\lor t) : A\lor B
  \end{prooftree}\qquad
  \begin{prooftree}
    t:A \justifies (t\lor *) : A\lor B
  \end{prooftree}
  \qquad
  \begin{prooftree}
    \justifies \var{x} : [p]
  \end{prooftree}
  \]\vspace{0.7em}
  \[
  \begin{prooftree}
    t : A \quad s: B \justifies (t \lor s) : A \lor B
  \end{prooftree}
  \qquad
  \begin{prooftree}t : A \quad s: B \justifies t \otimes s : A \otimes
    B \end{prooftree}\]\vspace{0.7em}
  \[
  \begin{prooftree} w_1 : [p] \ \cdots \ w_n:[p] \justifies (w_1+
    \cdots + w_n) : p \end{prooftree} \qquad
  \begin{prooftree} w_1 : [\bar{p}] \ \cdots \ w_n:[\bar{p}]
    \justifies (w_1+ \cdots + w_n) : \bar{p} \end{prooftree}\qquad
  \begin{prooftree} w_1 : A\otimes B \ \cdots \ w_n:A \otimes B
    \justifies (w_1+ \cdots + w_n) : A \land B \end{prooftree}
  \] \vspace{0.7em}
  \noindent\hrulefill
  \caption{Typing derivations for propositional expansion trees}
  \label{fig:typing-derivations}
\end{figure}

\begin{definition}
  A \emph{typed forest} is a finite forest $F$ of typed propositional
  expansion trees and witnesses, in which axiom variables occur in
  dual pairs, i.e.
  \begin{enumerate}
  \item each axiom variable $\var{x}$, and each negated variable
    $\bar{y}$, occurs at most once in $F$, and
  \item there is an occurrence of $\var{\bar{x}}$ in $F$ if and only
    if there is an occurrence of $\var{x}$.
  \end{enumerate}
  The type of a typed forest $F$ is the forest of types of the terms
  in $F$.  We will say that $F$ is an $e$-annotated sequent if all the
  terms in $F$ are expansion-trees: equivalently, if the type of $F$
  is a sequent of classical propositional logic (that is, it contains
  no witness types).
\end{definition}

\begin{example}
  The forest consisting of the three typed propositional expansion trees shown in 
  Example~\ref{ex:forestexample} is an e-annotated sequent. 
\end{example}

\begin{example}
  The following is a typed forest:
\[ ((\bar{w}) \ox (\bar{v})): \bar{p} \land \bar{p}, \ w:[p], \ v:[p]\]
\noindent It is not an e-annotated sequent, since some of its roots
are witnesses.
\end{example}

The e-annotated sequents are our notion of \emph{proof-structure}; the
more general notion of typed forests is needed to study subproofs and
cut-elimination.

\begin{example}
  The following e-annotated sequent arises by annotating the standard
  proof of Pierce's law
  \begin{equation}
    (((\var{\bar{x}}) \lor * ) \otimes (\var{\bar{y}})):(\bar{p} \lor
    q)\land \bar{p}, \quad  (\var{x} +
    \var{y}): p
  \end{equation}
  As with cw-annotated sequents, we can consider the \emph{graph} of
  this annotated sequent by adding in the axiom wires, giving a
  representation of our proof-structures closer to that usually seen
  for proof-nets:

\begin{definition}
  The \emph{graph} of an e-annotated sequent $F$ is a directed graph
  with vertices identical to the nodes of the forest of $F$.  The
  edges of the graph are given by the forest structure (with edges
  directed toward the root), plus an edge from $\var{x}$ to
  $\var{\bar{x}}$ for each wire variable $\var{x}$ appearing in $F$.
\end{definition}

For example, this graph represents the proof of Pierce's formula given
above:

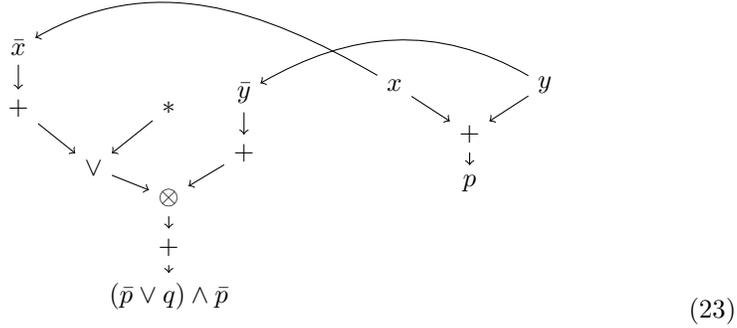
\begin{equation}
  \begin{tikzpicture}[level distance=6.5mm, grow=up, sibling
    distance=20mm, edge from parent path= {[<-](\tikzparentnode) to
      (\tikzchildnode)}]] 
    \node at (-2, 0) {$(\bar{p} \lor q)\land \bar{p}$} 
    child {node {$+$} 
      child {node{$\otimes$}  
        child [level distance=6mm]{ node {$+$} 
          child [level distance=8mm]{node (bary) {$\var{\bar{y}}$}}}
        child [level distance=4mm]{node{$\lor$} 
          child [level distance=8mm]{ node {$*$}}
          child [level distance=8mm] { node {$+$} 
            child {node (barx) {$\var{\bar{x}}$}}}}
      }};
    \node at (2,1.5) {$p$} 
    child {node {$+$} 
      child {node (y) {$\var{y}$}} 
      child {node (x) {$\var{x}$}}
    };
    \draw [->] (x) [in=30, out=150] to (barx);
    \draw [->] (y) [in=30, out=150] to (bary);
  \end{tikzpicture}\label{eq:1}
\end{equation}

\end{example}

The $e$-annotated sequents are our notion of proof structure: the
expansion nets are those $e$-annotated sequents which arise from
sequent proofs in $\LKvar$.  The procedure of inductively constructing
a proof-net from a sequent proof is given via the annotated sequent
calculus shown in Figure~\ref{fig:LKE}.

\begin{definition}
\label{def:exnets}
  An \emph{expansion-net} is an e-annotated sequent derivable in the
  system shown in Figure~\ref{fig:LKE}.
\end{definition}
\begin{remark}
  Notice that the order in which contractions occur in the sequent
  proof is no longer relevant to the net derived, as it was in na\"ive
  classical nets, since we represent contractions by the formal sum of
  witnesses. This can be seen in the following two examples of
  annotated derivations:
\end{remark}
 \begin{equation*}
  \begin{prooftree}
    \[
       \[
          \[
             \[ 
             \justifies 
             (\var{\bar{x}}):\bar{a}, \  (\var{x}):a
             \using \Ax \]
             \qquad 
             \[
             \justifies
             (\var{\bar{y}}):\bar{a}, \ (\var{y}):a
             \using \Ax \]
             \justifies (\var{\bar{x}}):\bar{a}, \ (\var{\bar{y}}):\bar{a},
             (\var{x} \otimes \var{y}): a \land a
             \using \land
          \] \qquad 
               \[
               \justifies (\var{\bar{z}}):\bar{a}, \ (\var{z}):a \using
               \Ax \] 
          \justifies (\var{\bar{x}}):\bar{a},\
               (\var{\bar{y}}):\bar{a},\ \var{\bar{z}}:\bar{a} ,\
               ((\var{x} \otimes \var{y}) \otimes \var{z}): (a \land
               a) \land a 
          \using \land
          \]
    \justifies 
         (\var{\bar{x}}):\bar{a}, \
         (\var{\bar{y}}+\var{\bar{z}}):\bar{a} ,\
         ((\var{x} \otimes \var{y}) \otimes \var{z}): (a \land a) \land a
    \using \C
    \]
    \justifies
    (\var{\bar{x}} +\var{\bar{y}}+\var{\bar{z}}): \bar{a}, \ ((\var{x} \otimes
    \var{y}) \otimes \var{z}): (a \land a) \land a
    \using \C
  \end{prooftree}
\end{equation*}

 \begin{equation*}
  \begin{prooftree}
    \[
       \[
          \[
             \[ 
             \justifies 
             (\var{\bar{x}}):\bar{a}, \  (\var{x}):a
             \using \Ax \]
             \qquad 
             \[
             \justifies
             (\var{\bar{y}}):\bar{a}, \ (\var{y}):a
             \using \Ax \]
             \justifies (\var{\bar{x}}):\bar{a}, \ (\var{\bar{y}}):\bar{a},
             (\var{x} \otimes \var{y}): a \land a
             \using \land
          \] \qquad 
               \[
               \justifies (\var{\bar{z}}):\bar{a}, \ (\var{z}):a \using
               \Ax \] 
          \justifies (\var{\bar{x}}):\bar{a},\
               (\var{\bar{y}}):\bar{a},\ \var{\bar{z}}:\bar{a} ,\
               ((\var{x} \otimes \var{y}) \otimes \var{z}): (a \land
               a) \land a 
          \using \land
          \]
    \justifies 
         (\var{\bar{x}}+\var{\bar{z}}):\bar{a}, \
         (\var{\bar{y}}):\bar{a} ,\
         ((\var{x} \otimes \var{y}) \otimes \var{z}): (a \land a) \land a
    \using \C
    \]
    \justifies
    (\var{\bar{x}} +\var{\bar{y}}+\var{\bar{z}}): \bar{a}, \ ((\var{x} \otimes
    \var{y}) \otimes \var{z}): (a \land a) \land a
    \using \C
  \end{prooftree}
\end{equation*}

\begin{landscape}
\begin{figure}
\centering
\[
\begin{prooftree}
\[
\[
\[
\[
\[
\[
\justifies
(\bar{x}):\bar{q}, (x):q
\using \Ax
\]  
\justifies
(\bar{x}):\bar{q}, (* \lor (x)):(q \lor q)
\using \lor_1
\]  \quad 
\[\justifies  (w):p, (\bar{w}):\bar{p} \using \Ax\]
\justifies
(\bar{x}):\bar{q}, t:((q \lor q)\land p), (\bar{w}):\bar{p}
\using
\land
\]
\[
\[
\[
\justifies
(\bar{y}):\bar{q}, (y):q
\using \Ax
\]  
\justifies
(\bar{y}):\bar{q}, ((y)\lor *): (q \lor q)
\using \lor_0
\]  \quad 
\[\justifies  (v):p, (\bar{v}):\bar{p} \using \Ax\]
\justifies
(\bar{y}):\bar{q}, s:((q \lor q)\land p), (\bar{v}):\bar{p}
\using
\land
\]
\justifies 
((\bar{w})\otimes(\bar{v})):(\bar{p} \land \bar{p}), (\bar{x}):\bar{q}, (\bar{y}):\bar{q}, t:((q \lor q)\land p), s:((q \lor
q)\land p)
\using \land
\]
\justifies
((\bar{w})\otimes(\bar{v}) ):(\bar{p} \land \bar{p}), \
(\bar{x}+\bar{y}) :\bar{q}, \ s+t:((q \lor q)\land p)
\using \C^2
\]\ \hspace{-6em} \[ \justifies (z):q, (\bar{z}):\bar{q} \using\Ax \]
\justifies
(\bar{x}+\bar{y}):\bar{q}, \ (\bar{z}):\bar{q},
((\bar{w})\otimes(\bar{v}) ):(\bar{p}\land \bar{p}), \ s+t:(((q \lor q) \land p)\land
q)
\using \land
\]
  \justifies
(\bar{x}+\bar{y}+\bar{z}):\bar{q},\ ((\bar{w})\otimes(\bar{v}) )
:(\bar{p}\land \bar{p}),\ ((s+t) \ox (z)):(((q \lor q) \land p)\land q)
\using \C
\end{prooftree}
\]
\caption{A sample derivation in annotated $\LKvar$}
\label{fig:derivex}
\end{figure}
\end{landscape}
\begin{example}
  The e-annotated sequent 
  \[ F =  (\bar{x} + \bar{y} + \bar{z}):\bar{q}, \quad ((\bar{w}) \ox
  (\bar{v})):(\bar{p}\land \bar{p}),\]\[ \quad ((((* \lor (x)) \ox (v))+
  (((y) \lor *)\ox (w))) \ox (z)):(((q \lor q) \land p)\land q)  \] 
is an expansion-net: if we let $t = (((y) \lor *)\ox (w))$, and $s=((*
\lor (x)) \ox (v))$, then the derivation in Figure~\ref{fig:derivex}
is a derivation of $F$. 

\end{example}
Cut-free formula-completeness of $\LKvar$ gives us the following:

\begin{theorem}
  A formula $A$ of classical propositional logic is valid if and only
  if there is an expansion net $t:A$.
\end{theorem}

\begin{figure}[t]
  \noindent\hrulefill
  \centering
  \[
  \begin{prooftree}
    \justifies 1:\top \using \Ax_\top
  \end{prooftree} \qquad
  \begin{prooftree}
    F \quad G \justifies F, \ G \using \mix
  \end{prooftree}
  \qquad
  \begin{prooftree}
    \justifies (\var{\bar{x}}): \bar{p}, \ (\var{x}): p \using \Ax
  \end{prooftree}
  \]\vspace{0.7em}
  \[
  \begin{prooftree}
    F, \ t : A, \ s:B \justifies F, \ t\lor s : A \lor B \using \lor
  \end{prooftree} \qquad
  \begin{prooftree}
    F, \ t : A \justifies F, \ t\lor * : A \lor B \using \lor_0
  \end{prooftree}
  \quad
  \begin{prooftree}
    F, \ s : B \justifies F, \ *\lor s : A \lor B \using \lor_1
  \end{prooftree} \qquad\] \vspace{0.7em}
  \[
  \begin{prooftree}
    F, \ t : A \qquad G, \ s: B \justifies F, G,\ ( t\otimes s ) : A
    \land B \using \land
  \end{prooftree} \] \vspace{0.7em}
  \[
  \begin{prooftree}
    F, \ t: A \land B, \ s: A \land B \justifies F, \ t + s : A \land
    B \using \C_\land
  \end{prooftree}\qquad\qquad
  \begin{prooftree}
    F, \ s: p, \ t: p \justifies F, \ s+t : p \using \C_p
  \end{prooftree}
  \qquad\qquad
  \begin{prooftree}
    F, \ s: \bar{p}, \ t: \bar{p} \justifies F, \ s + t : \bar{p}
    \using \C_{\bar{p}}
  \end{prooftree}
  \]
  \noindent\hrulefill
  \caption{$\LKE$: an annotated version of $\LKvar$ deriving expansion
    nets}
  \label{fig:LKE}
\end{figure}

Each e-annotated sequent corresponds to an equivalence-class of
default-attached cw-annotated sequents, modulo the associativity and
commutativity of contraction and the pointwise construction of
contractions.  Furthermore, it is easy to verify that, given an
equivalence class of cw-annotated sequents induced by an e-annotated
sequent, either all or none of them are correct.  Thus, correctness of
a member of the equivalence class can be used to define a notion of
correctness for expansion-nets.  However, it will be useful later to
consider the idea of a switching path in an expansion-net, and for
this reason we give now an independent definition of correctness for
expansion-nets -- actually, for all typed forests.  We give here the
notion of AC-correctness (AC for ACyclic, as distinct from ACC,
ACyclic and Connected, the usual criterion for \MLLminus nets) for
typed forests:

\begin{definition} Let $F$ be a typed forest
\label{def:correctness}
  \begin{enumerate}
  \item A node $X$ of $F$ is a \emph{switched node} if it is an
    expansion, or if it is a $\lor$ node $t \lor s$ where neither $t$
    nor $s$ is an instance of $*$.
  \item A switching $\sigma$ for $F$ is a choice of successor for each
    switched node.
  \item The switching graph $\sigma(F)$ is obtained from the graph of
    $F$ by:
    \begin{enumerate}
    \item[1:] deleting all incoming edges to each switched node other
      than those coming from the nodes chosen by the switching, and
    \item[2:]forgetting the directedness of edges.
    \end{enumerate}
  \item $F$ is $AC$-correct if, for every switching $\sigma$ of $F$,
    $\sigma(F)$ is acyclic.
  \end{enumerate}
\end{definition}
\begin{remark}
  Notice that nodes of the form $(t \lor *)$ and $(* \lor t)$ are
  \emph{unswitched}; we can see this as implicitly adding to the
  switching graph an attachment from $*$ to $t$.
\end{remark}

While the switching graph definition of correctness suggests an
exponential-time correctness algorithm, it is essentially the same as
the \MLL + \mix\ switching criterion, and can therefore be checked in
polynomial time; such a polynomial time algorithm is given, for
example, by attempting to sequentialize by searching for
\emph{splitting pars}, a technique first described
in~\cite{Dan90Thesis}, and available in English translation in the
Linear Logic Primer~\cite{LinLogPrimer}.  

A useful notion arising from the switching graphs is that of a
\emph{switching path}: a nonempty sequence $P = X_1, \dots X_n$ of
nodes of $F$ which defines a path in some switching graph $F\sigma$ of
$F$.  The AC correctness criterion can, using this notion, be stated
as follows: an annotated sequent is correct if it all its switching
paths are acyclic.  We will refer to a switching path as ``entering a
node $X$ through its successor $Y$'' (or ``entering $X$ from above'')
on a switching path $P$ if the node $Y$ is immediately followed by the
node $X$ in $P$, and ``entering a node $X$ through its predecessor
$Z$'' (or ``entering $X$ from below'') if $X$ is immediately preceded
by $Z$ in $P$.  Terminology related to a path leaving a node through
predecessors/successor is defined analogously.

The $AC$ correctness criterion characterizes, of course, those
e-annotated sequents derivable in $\LKE$:

\begin{theorem}
\label{thm:correctness}
  An e-annotated sequent $F$ is an expansion-net (i.e. $\LKE
  \turnstile F$) if and only if $F$ is $AC$-correct.
\end{theorem}

This result can be proved via a number of techniques, including the
aforementioned ``splitting pars'' technique, or the earlier
``splitting tensors'' technique.  The latter was adapted for \MLL +
\mix\ by Bellin in~\cite{Bel97SubMix}.  In Section~\ref{sec:subnets}
below, we give a proof directly for expansion-nets which uses the new
notion of a contiguous subnet.

\subsection{Comparison with other notions of invariant}
It should be clear that expansion nets identify more proofs of
$\LKvar$ than na\"ive classical nets: two proofs differing only by the
order in which contractions are performed will have different na\"ive
nets but the same expansion net.  We take some time now to compare the
equivalence classes of proofs defined by expansion nets and the other
existing notions of abstract proof -- $\mathbb{N}$-nets and
combinatorial proofs (since $\mathbb{B}$-nets identify strictly more
proofs than $\mathbb{N}$-nets, we will not consider them further).  

\subsubsection{$\mathbb{N}$-nets identify more $\LK^*$ derivations
  than expansion nets}
To obtain an $\mathbb{N}$-net from a given derivation, one simply
traces paths from positive to negative atoms in the conclusion of the
proof.  For both $\LK^*$ derivations and expansion-nets, there is an
obvious way to this:  and it is not difficult to establish he
following:
\begin{proposition}
  Let $\Phi$ be an $\LK^*$ derivation, and let $F$ be its
  corresponding expansion net.  The $\mathbb{N}$-nets of $\Phi$ and
  $F$ coincide.
\end{proposition}

This means that expansion-nets cannot distinguish two proofs
identified by their $\mathbb{N}$-nets; said differently, expansion
nets contain at least as much information as $\mathbb{N}$-nets.  In
fact, they contain strictly more information.  Consider the following
two sequent derivations proving the same sequent:
\[
\begin{prooftree}
  \[
  \[
  \[
  \bar{a}, a \justifies \bar{a} \lor \bar{b}, a \using \lor_0
  \] \qquad
  \[\bar{c}, c
  \justifies \bar{c} \lor \bar{d}, c \using \lor_0
  \]
  \justifies (\bar{a} \lor \bar{b}) \land (\bar{c} \lor \bar{d}), a
  \lor c \using \land,\lor
  \]
  \qquad
  \[
  \[\bar{b}, b
  \justifies \bar{a} \lor \bar{b} , b \using \lor_1\] \qquad
  \[\bar{d}. d
  \justifies \bar{c} \lor \bar{d}, d \using \lor_1\] \justifies
  (\bar{a} \lor \bar{b}) \land (\bar{c} \lor \bar{d}), b \lor d \using
  \land,\lor
  \]
  \justifies (\bar{a} \lor \bar{b}) \land (\bar{c} \lor \bar{d}),\
  (\bar{a} \lor \bar{b}) \land (\bar{c} \lor \bar{d}) , \ (a \lor c)
  \land (b \lor d) \using \land
  \]
  \justifies (\bar{a} \lor \bar{b}) \land (\bar{c} \lor \bar{d}), (a
  \lor c) \land (b \lor d) \using \C
\end{prooftree}
\]

\[
\begin{prooftree}
  \[
  \[
  \[
  a, \bar{a} \justifies a \lor c, \bar{a} \using \lor_0
  \] \qquad
  \[b, \bar{b} \justifies b \lor d, \bar{b} \using \lor_0
  \]
  \justifies (a \lor c) \land (b \lor d), \bar{a} \lor \bar{b} \using
  \land,\lor
  \]
  \qquad
  \[
  \[\bar{c}, c
  \justifies a \lor c , \bar{c} \using \lor_1\] \qquad
  \[d, \bar{d} \justifies b \lor d, \bar{d} \using \lor_1\] \justifies
  (a \lor c) \land (b \lor d), (\bar{c} \lor \bar{d}) \using
  \land,\lor
  \]
  \justifies (\bar{a} \lor \bar{b}) \land (\bar{c} \lor \bar{d}) , \
  (a \lor c) \land (b \lor d), (a \lor c) \land (b \lor d) \using
  \land
  \]
  \justifies (\bar{a} \lor \bar{b}) \land (\bar{c} \lor \bar{d}), (a
  \lor c) \land (b \lor d) \using \C
\end{prooftree}
\]

These two proofs have the same $\mathbb{N}$-net, with one link between
each pair of dual atoms, but different expansion nets:
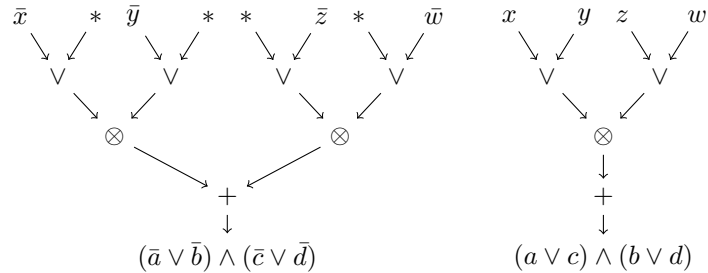
\[
\begin{tikzpicture}[level distance=8mm, grow=up, sibling
    distance=30mm, edge from parent path= {[<-](\tikzparentnode) to
      (\tikzchildnode)}] 
 \node at (-4, 0) {$(\bar{a} \lor \bar{b}) \land (\bar{c} \lor \bar{d})$} 
    child {node {$+$} child{node {$\ox$} child [sibling
    distance=15mm]{node {$\lor$}
                                                     child [sibling
    distance=10mm] {node
                                                       {$\bar{w}$}} 
                                                     child [sibling
    distance=10mm] {node{{$*$}}}} 
                                          child [sibling
    distance=15mm] {node {$\lor$} child [sibling
    distance=10mm]
                                                {node {$\bar{z}$}} child [sibling
    distance=10mm] {node{$*$}}}}
child {node {$\ox$} child [sibling
    distance=15mm]{node {$\lor$}
                                                     child [sibling
    distance=10mm]{node
                                                       {$*$}} 
                                                     child [sibling
    distance=10mm] {node{$\bar{y}$}}} 
                                          child [sibling
    distance=15mm] {node {$\lor$} child [sibling
    distance=10mm]
                                                {node {$*$}} child [sibling
    distance=10mm] {node{$\bar{x}$}}}}};
\node at (1,0) {$(a
  \lor c) \land (b \lor d)$} child { node {$+$} child {node {$\ox$} child [sibling
    distance=15mm]{node {$\lor$}   child [sibling
                                     distance=10mm]{node {$w$}} 
                                   child [sibling
                                     distance=10mm] {node{$z$}}}
                       child [sibling
    distance=15mm] {node {$\lor$}  child [sibling
                                     distance=10mm] {node {$y$}} 
                       child [sibling
                                     distance=10mm] {node{$x$}}}}};
\end{tikzpicture}
\]

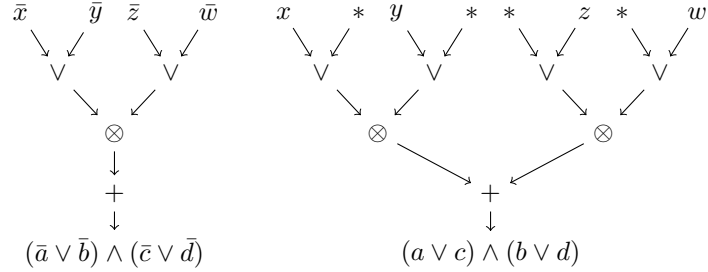
\[
\begin{tikzpicture}[level distance=8mm, grow=up, sibling
    distance=30mm, edge from parent path= {[<-](\tikzparentnode) to
      (\tikzchildnode)}] 
 \node at (4, 0) {$(a \lor c) \land (b \lor d)$} 
    child {node {$+$} child{node {$\ox$} child [sibling
    distance=15mm]{node {$\lor$}
                                                     child [sibling
    distance=10mm] {node
                                                       {$w$}} 
                                                     child [sibling
    distance=10mm] {node{{$*$}}}} 
                                          child [sibling
    distance=15mm] {node {$\lor$} child [sibling
    distance=10mm]
                                                {node {$z$}} child [sibling
    distance=10mm] {node{$*$}}}}
child {node {$\ox$} child [sibling
    distance=15mm]{node {$\lor$}
                                                     child [sibling
    distance=10mm]{node
                                                       {$*$}} 
                                                     child [sibling
    distance=10mm] {node{$y$}}} 
                                          child [sibling
    distance=15mm] {node {$\lor$} child [sibling
    distance=10mm]
                                                {node {$*$}} child [sibling
    distance=10mm] {node{$x$}}}}};
\node at (-1,0) {$(\bar{a}
  \lor \bar{b}) \land (\bar{c} \lor \bar{d})$} child { node {$+$} child {node {$\ox$} child [sibling
    distance=15mm]{node {$\lor$}   child [sibling
                                     distance=10mm]{node {$\bar{w}$}} 
                                   child [sibling
                                     distance=10mm] {node{$\bar{z}$}}}
                       child [sibling
    distance=15mm] {node {$\lor$}  child [sibling
                                     distance=10mm] {node {$\bar{y}$}} 
                       child [sibling
                                     distance=10mm] {node{$\bar{x}$}}}}};
\end{tikzpicture}
\]

Identifying these two proofs, as suggested by their $\mathbb{N}$-nets, does not seem
at all natural in the multiplicatively formulated sequent calculus (it
arises very naturally, however, in the \emph{deep inference} proof
theory of classical logic
\cite{BruLCL06,BruTh}, which
provided inspiration for the design of $\mathbb{N}$-nets.) In light of
this, and the sequentialization theorem, we claim that expansion-nets
provide a better notion of abstract proof for sequent proofs than
$\mathbb{N}$-nets.

\subsubsection{Combinatorial proofs identify at least as many $\LKvar$
  derivations as expansion nets}

To see how to extract a combinatorial proof from an expansion-net,
we will need the following intuitive notion: an
expansion tree $F$ of type $\G$ induces a function $f$ from the wire
variables of $F$ to the leaves (atom occurrences) of $\G$.  This
function arises in much the same way as the $\mathbb{N}$-net of an
expansion-net: by tracing the atoms through the tree.   
Given an expansion net $F$, extract a co-graph
from $F$ as follows: the vertices of the co-graph are the wire
variables of $F$, and there is an edge between two wire variables if
and only if smallest subtree of $F$ containing both variables is an
$\ox$ tree.  The function $f$ from wire variables to atoms is a
contraction-weakening,
by the structure of propositional expansion trees, and so the pair of
co-graph and function given by an expansion-net defines a correct
combinatorial proof.  For example, the expansion-net in Example~\ref{ex:forestexample} yields the
(semi-)combinatorial proof in Example~\ref{ex:combproof}.

This
combinatorial proof is the same proof as would be extracted directly
from an $\LKvar$ derivation giving rise to $F$: thus combinatorial
proofs identify, at the very least, all the proofs identified by
expansion-nets.  It is likely that, in fact, combinatorial proofs
identify the same $\LKvar$ derivations as expansion-nets; if so, this would
provide a criterion identifying just those combinatorial proofs
arising from $\LKvar$ derivations.

\section{Subnets of expansion nets}
\label{sec:subnets}
 In the sequent
calculus, we have a clear notion of ``subproof of a sequent proof'',
given by subtrees.
In proof-nets,
it is harder to see, intuitively, the correct notion of subproof, and this
causes a number of conceptual problems when manipulating proofs. The
notion of ``subnet'' captures, in proof nets, the concept of subproof.

Subproofs play two important roles in the proof theory of classical
logic.  The first is that proving cut-elimination often relies on a
principal lemma in  which it is shown that a single cut can be
eliminated from an otherwise cut-free proof: in this case the cut is
always the final rule in the proof.  Full cut-elimination then follows
by considering uppermost cuts: the subproof introducing an uppermost
cut contains no other cuts.  In proof nets, there is
no clear notion of uppermost cut, or of the subproof containing a cut.
It is with a view to obtaining such a notion that we define the
subnets of a net.
 
The second role that subproofs play is in the definition of
cut-reduction steps, where one of the cut-formulae is the result of a
structural rule.  For example, the usual way to reduce a cut against
contraction, such as
\begin{equation}
  \begin{prooftree}
    \[\leadsto
    F,  A \using \Phi
    \]
    \qquad
    \[
    \[
    \leadsto \quad F', \bar{A}, \bar{A}\quad \using
    \Psi
    \]
    \justifies F, \bar{A} \using C \] \using \Cut
    \justifies F, F'
  \end{prooftree}
  \label{eq:structredex}
\end{equation}
\noindent is to duplicate the subproof $\Phi$, and then contract the
resulting duplicated conclusions:
\begin{equation}
  \begin{prooftree}
\[
\[
    \leadsto
    F,  A \using \Phi
    \]
\qquad 
    \[
        \[
        \leadsto
         F,  A \using \Phi
         \]
    \qquad
         \[
         \leadsto F' \bar{A}, \bar{A} \qquad \using
         \Psi
         \]
    \justifies F, F', \bar{A}  \using \Cut 
    \] 
\using \Cut
\justifies F,F, F'
\]
\justifies
F, F'
\using \C^*
  \end{prooftree}
  \label{eq:structredex}
\end{equation}

To perform such an operation in proof-nets requires that we know what
a subproof is, and can find them.  In linear logic proof nets with
exponentials, duplication and deletion are typically mediated by
\emph{boxes} -- that is, the subgraphs to be duplicated are explicitly
marked regions of the net.  Expansion-nets are box-free, and so the
appropriate subgraph to delete or duplicate must be calculated;
further, we must ensure that this duplication or deletion does not
break correctness. 

A \emph{subnet} of an
expansion-net $F$ (a concept first introduced for \MLLminus nets
in~\cite{212898}) is a graph corresponding to a subproof of $F$: we
define them as follows:


\begin{definition}
  Let $F$ be a typed forest: a substructure of $F$ is a subforest
  $G$ of $F$ which is 
  \begin{itemize}
  \item closed under axiom links: that is, if the leaf annotated with
    $x$ is in $G$, then so is the leaf annotated with $\bar{x}$.
\item closed under default attachment: that is, if an instance of $*$
  occurs in $G$, then its predecessor $(t \lor *)$ or $(* \lor t)$ is
  in $G$.
  \end{itemize}
  If $X$ is a node of $F$, let $\str(X)$ be the smallest substructure
  of $F$ containing $X$.
\end{definition}


\begin{definition}
\label{def:subnet}
  Let $F$ be an $AC$ typed forest: a \emph{subnet} of $F$ is a
  substructure $G$ of $F$ such that, 
  for any two roots $X$, $Y$ of $G$, every switching path
   between $X$ and $Y$ lies inside $G$.
\end{definition}

  A more obvious (but incorrect) notion of subnet for expansion-nets
  would be, simply, a subforest which is, itself, an expansion net.
  This simplistic kind of definition works for $\MLLminus$ proof-nets,
  for example.   Consider, however, the following sequent proof in classical logic:

\[\pi = 
\begin{prooftree}
\[
\[
\[
\[
\[ \justifies \bar{p}, p \using \Ax\] \quad \[  \justifies \bar{p}, p \using \Ax \]
\justifies 
(\bar{p} \land \bar{p}), p, p
\using \C
\]
\justifies
(\bar{p} \land \bar{p}), p
\using
\C
\]
\quad\[ \justifies \bar{p}, \quad p  \using \Ax \]
\justifies
(\bar{p} \land \bar{p}), \bar{p}, \quad p, p
\using \mix
\]
\justifies
(\bar{p} \land \bar{p}) \lor \bar{p}, \quad p, p 
\using \lor
\]
  \justifies 
(\bar{p}\land \bar{p})\lor\bar{p},  \quad p  
\using \C
\end{prooftree}
\]

The expansion net $F$ corresponding to $\pi$ is:

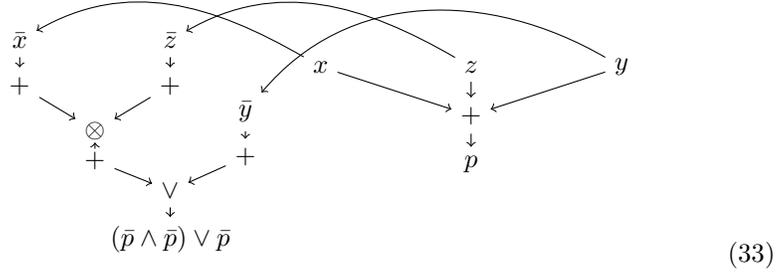
\begin{equation}
  \begin{tikzpicture}[level distance=6.5mm, grow=up, sibling
    distance=20mm, edge from parent path= {[<-](\tikzparentnode) to
      (\tikzchildnode)}]] 
    \node at (-2, 0) {$(\bar{p} \land \bar{p})\lor \bar{p}$} 
      child {node{$\lor$}  
          child [level distance=4.5mm]{ node {$+$} 
            child [level distance=6mm]{node (bary) {$\var{\bar{y}}$}}}
        child [level distance=4mm]{node {$+$} child {node{$\ox$} 
            child [level distance=6mm]{ node {$+$} 
              child {node (barz) {$\var{\bar{z}}$}}}
            child [level distance=6mm] { node {$+$} 
              child {node (barx) {$\var{\bar{x}}$}}}}}
      };
    \node at (2,1) {$p$} 
    child {node {$+$} 
      child {node (y) {$\var{y}$}}
      child {node (z) {$\var{z}$}}
      child {node (x) {$\var{x}$}}
    };
    \draw [->] (x) [in=30, out=150] to (barx);
    \draw [->] (y) [in=50, out=150] to (bary);
     \draw [->] (z) [in=30, out=150] to (barz);
\end{tikzpicture}\label{eq:1}
\end{equation}

Now consider the sub-proof of the sequent proof proving $p\land p, p, p$.
The expansion-net corresponding to that proof is the following:
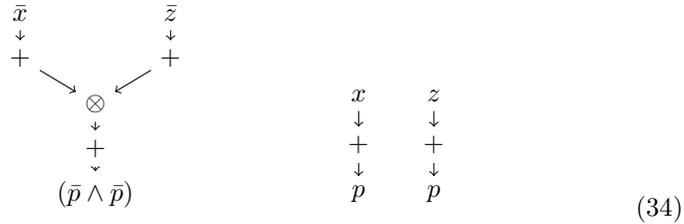
\begin{equation}
  \begin{tikzpicture}[level distance=6.5mm, grow=up, sibling
    distance=20mm, edge from parent path= {[<-](\tikzparentnode) to
      (\tikzchildnode)}] 
    \node at (-2, 0) {$(\bar{p} \land \bar{p})$} 
        child [level distance=6mm]{node {$+$} child {node{$\ox$}
            child [level distance=6mm]{ node {$+$} 
              child {node (barz) {$\var{\bar{z}}$}}}
            child [level distance=6mm] { node {$+$} 
              child {node (barx) {$\var{\bar{x}}$}}}}};
    \node at (1.5,0) {$p$} 
    child {node {$+$} 
      child {node (y) {$\var{x}$}}};
\node at (2.5,0) {$p$} 
    child {node {$+$} 
      child {node (z) {$\var{z}$}}};
\end{tikzpicture}\label{eq:9}
\end{equation}

This does not appear as a subforest of $F$; in other words,
the subforests of $F$ which are themselves expansion-nets do not 
suffice to express the sub-proofs of $F$.  The subnet corresponding to 
the subproof is in this case \emph{not} an expansion net: it is the
shaded subgraph in the following:
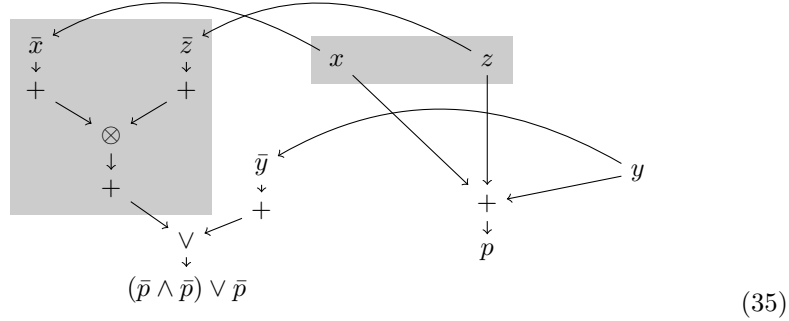
\begin{equation}
  \begin{tikzpicture}[level distance=6.5mm, grow=up, sibling
    distance=20mm, edge from parent path= {[<-](\tikzparentnode) to
      (\tikzchildnode)}]] 
    \node at (-2, 0) {$(\bar{p} \land \bar{p})\lor \bar{p}$} 
      child {node{$\lor$}  
          child [level distance=4mm]{ node {$+$} 
            child [level distance=6mm]{node (bary) {$\var{\bar{y}}$}}}
        child [level distance=7mm]{node (plus){$+$} child {node{$\ox$} 
            child [level distance=6mm]{ node {$+$} 
              child {node (barz) {$\var{\bar{z}}$}}}
            child [level distance=6mm] { node {$+$} 
              child {node (barx) {$\var{\bar{x}}$}}}}}
      };
    \node at (2,0.5) {$p$} 
    child {node {$+$} 
      child [level distance=4mm]{node (y) {$\var{y}$}}
      child [level distance=19mm] {node (z) {$\var{z}$}}
      child [level distance=19mm] {node (x) {$\var{x}$}}
    };
    \draw [->] (x) [in=30, out=150] to (barx);
    \draw [->] (y) [in=30, out=150] to (bary);
     \draw [->] (z) [in=30, out=150] to (barz);
 \begin{pgfonlayer}{background}
       \node [fill=black!20,fit=(plus) (barz) (barx)] {}; 
        \node [fill=black!20,fit=(z) (x)] {}; 
     \end{pgfonlayer}
\end{tikzpicture}\label{eq:7}
\end{equation}
\noindent or, alternatively, $(\bar{x}+\bar{y}):\bar{p} \land \bar{p},
x:[p] , y:[p]$, which is not an expansion net, as it has witnesses as
roots.

Now consider the following shaded substructure of $F$, which is \emph{not} a
subnet of $F$:

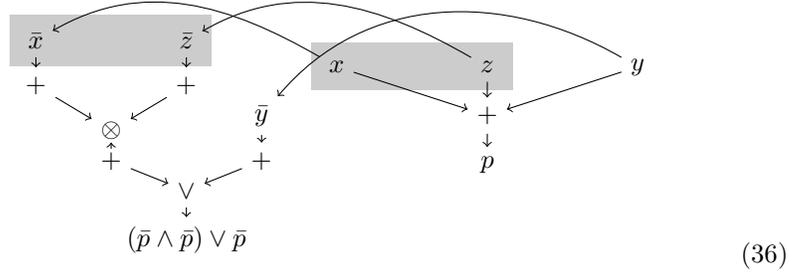
\begin{equation}
  \begin{tikzpicture}[level distance=6.5mm, grow=up, sibling
    distance=20mm, edge from parent path= {[<-](\tikzparentnode) to
      (\tikzchildnode)}]] 
    \node at (-2, 0) {$(\bar{p} \land \bar{p})\lor \bar{p}$} 
      child {node{$\lor$}  
          child [level distance=4mm]{ node {$+$} 
            child [level distance=6mm]{node (bary) {$\var{\bar{y}}$}}}
        child [level distance=4mm]{node (plus){$+$} child {node{$\ox$} 
            child [level distance=6mm]{ node {$+$} 
              child {node (barz) {$\var{\bar{z}}$}}}
            child [level distance=6mm] { node {$+$} 
              child {node (barx) {$\var{\bar{x}}$}}}}}
      };
    \node at (2,1) {$p$} 
    child {node {$+$} 
      child {node (y) {$\var{y}$}}
      child {node (z) {$\var{z}$}}
      child {node (x) {$\var{x}$}}
    };
    \draw [->] (x) [in=30, out=150] to (barx);
    \draw [->] (y) [in=50, out=150] to (bary);
     \draw [->] (z) [in=30, out=150] to (barz);
 \begin{pgfonlayer}{background}
       \node [fill=black!20,fit=(barz) (barx)] {}; 
        \node [fill=black!20,fit=(z) (x)] {}; 
     \end{pgfonlayer}
\end{tikzpicture}\label{eq:11}
\end{equation}
This typed forest satisfies the AC correctness criterion: each of its
switching graphs is acyclic.  However, it is not a subnet of $F$,
since there is a switching path from $\bar{x}$ to $\bar{z}$ which
passes outside the shaded substructure.  This shaded substructure does
\emph{not} correspond to any subproof of $\pi$, nor of any other
sequentialization of $F$.  For more discussion on subnets in the
presence of the mix rule, see~\cite{Bel97SubMix}.



\subsection{Kingdoms and Empires}
\label{sec:kingdoms-empires}

Given a proper node $X$ (that is, a node which is not an instance of
$*$), the subnets with $X$ as a root correspond to subproofs with $X$
in the conclusion.  

 

Given any proper node $X$, the set of subnets with $X$ as a root are
closed under intersection:
\begin{lemma}
\label{lem:inserseckingdoms}
  Let $G_1$ and $G_2$ be subnets of an AC typed forest $F$ having the
  node $X$ as a root.  Let $G_1 \cap G_2$ denote the substructure of
  $F$ defined by the nodes of $F$ common to $G_1$ and $G_2$.  Then
  $G_1 \cap G_2$ is a subnet of $F$.
\end{lemma}
\begin{proof}
  Suppose there is a switching path $P$ between two roots of $G_1 \cap
  G_2$ but outside of $G_1\cap G_2$.  If both $X$ and $Y$ are roots of
  $G_1$, then $G_1$ is not a subnet, similarly for $G_2$: therefore
  without loss of generality $X$ is only a root of $G_1$, and $Y$ only
  a root
  of $G_2$.  Since $Y$ is not a root of $G_1$, $P$ passes through some
  root of $G_1$: but then the path from that root to $X$ is a path
  between two roots of $G_1$, outside of $G_1$, and so $G_1$ is not a subnet.
 \end{proof}

This means, particular, that we can consider the smallest subnet with
$X$ as a root, given by taking the intersection of all such subnets : the
following terminology originates in~\cite{212898}.

\begin{definition}
  Let $F$ be an AC typed forest, and let be $X$ a node of $F$, such
  that at least one subnet of $F$ has $X$ as a root.  The kingdom
  $k(X)$ of $X$ in $F$ is the smallest subnet of $F$ which has $X$ as
  a root.
\end{definition}
\noindent Notice that, by this definition, only proper nodes can have
a kingdom or empire: there is no substructure of any expansion net
having a $*$ as a root.

\begin{example}
  The shaded net shown in example \ref{eq:7} is the kingdom
  of its leftmost root.
\end{example}

Kingdoms are of interest because they allow us to see additional dependencies
between nodes in an expansion-net.  If in an expansion-net
$F$ a node $Y$ is in the kingdom of a node $X$, then in every
sequentialization of F (every $\LKE$ derivation resulting in $F$) the
rule introducing $Y$ will occur in a subproof of the rule introducing
$X$.  We will use the relation symbol $\ll$ to denote this
\emph{kingdom ordering}:
\[ X \ll Y \mbox{ if and only if $X$ is in the kingdom of $Y$}.\]
\noindent This relation plays a key role in our proof of
cut-elimination for expansion-nets (Theorem~\ref{thm:cut-elim}).  It
allows us to recover a notion of ``uppermost cut'' in an expansion
net: a cut which is $\ll$-maximal corresponds to a cut which can be
sequentialized such that no other cut lies above it.


The relation $\ll$ also plays an important role in our proof of
sequentialization for expansion-nets (Theorem~\ref{thm:seq}).  In
fact, sequentialization is nothing more than the completion of the relation $\ll$
to a tree-relation on the nodes of an expansion net. We will need, in
the proof of the sequentialization theorem, the following fact: two
nodes of an AC typed forest have the same kingdom if and only if they
are a pair of dual wire variables.
\begin{proposition}
\label{prop:kingdomordering}
$\ll$ is a preorder on the proper nodes of an AC typed forest $F$, and
moreover is a partial order on the nonatomic proper nodes of $F$.
\end{proposition}
\begin{proof}
  The relation $\ll$ is clearly reflexive and transitive.  We show
  that it is antisymmetric if restricted to the nonatomic nodes of $F$.
  Suppose that $X$ and $Y$ are distinct nodes of $F$, and that $X \in
  k(Y)$ and $Y \in k(X)$.  Then clearly $k(X)=k(Y)$, since otherwise
  the intersection of $k(X)$ and $k(Y)$ would be a smaller subnet with
  both $X$ and $Y$ as roots.  This equality holds in the case where
  $X$ and $Y$ are dual wire variables: the two ends of a wire arising
  from an axiom link: we must now show that it cannot hold if either
  $X$ or $Y$ is nonatomic.  Suppose first that $X$ is a disjunction or
  nontrivial expansion; then by removing $X$ from $k(Y)$ (but keeping
  its successors) we find a smaller subnet with $Y$ as a root:
  contradiction.  Now suppose that $X = (X_1 \ox X_2)$.  Then $k(X) =
  k(X_1) \cup k(X_2) \cup \finset{X}$, and so $Y$ is a member of
  $k(X_i)$ for $i \in \finset{0,1}$.  Since $Y\in k(X_i)$ and $k(X_i)
  \in k(Y)$, we have as above that $k(X_i) = k(Y)$.  But $X \notin
  k(X_i)$; contradiction.
\end{proof}

We have not yet shown that every proper node of an expansion-net has a
kingdom.  Bellin shows directly in~\cite{Bel97SubMix} that every node
has a kingdom, but this is a rather difficult proof: for an easier
proof we turn now to the new notion of contiguousness.

\subsection{Contiguous subnets}
\label{sec:contiguous-subnets}
A natural counterpart to the notion of kingdom, the smallest subnet
with a given node as root, is the notion of empire:

\begin{definition}
  \label{def:empire}
Let $F$ be an AC-correct typed forest, and $X$ a proper node of $F$.
The \emph{empire} $e(X)$ of $X$ in $F$ is the largest subnet of $F$ with $X$ as a root.
\end{definition}

\begin{example}
Continuing our example from above, the shaded subnet in the following
is the empire of its leftmost root:
\begin{equation}
  \begin{tikzpicture}[level distance=6.5mm, grow=up, sibling
    distance=20mm, edge from parent path= {[<-](\tikzparentnode) to
      (\tikzchildnode)}]] 
    \node at (-2, 0) {$(\bar{p} \land \bar{p})\lor \bar{p}$} 
      child {node{$\lor$}  
          child [level distance=4mm]{ node (aplus){$+$} 
            child [level distance=6mm]{node (bary) {$\var{\bar{y}}$}}}
        child [level distance=6mm]{node (plus){$+$} child {node{$\ox$} 
            child [level distance=6mm]{ node {$+$} 
              child {node (barz) {$\var{\bar{z}}$}}}
            child [level distance=6mm] { node {$+$} 
              child {node (barx) {$\var{\bar{x}}$}}}}}
      };
    \node at (2,1) {$p$} 
    child {node {$+$} 
      child {node (y) {$\var{y}$}}
      child {node (z) {$\var{z}$}}
      child {node (x) {$\var{x}$}}
    };
    \draw [->] (x) [in=30, out=150] to (barx);
    \draw [->] (y) [in=50, out=150] to (bary);
     \draw [->] (z) [in=30, out=150] to (barz);
 \begin{pgfonlayer}{background}
       \node [fill=black!20,fit=(plus) (barz) (barx)] {}; 
        \node [fill=black!20,fit=(z) (x)(y)] {};
        \node [fill=black!20,fit=(aplus) (bary)] {};
     \end{pgfonlayer}
\end{tikzpicture}\label{eq:11}
\end{equation}

\end{example}

In the absence of the mix rule (that is, if we assume that every
switching graph is not only acyclic, but also connected), the empire
is a very useful concept: it is very easy to show that every proper
node of an AC-correct typed forest has an empire (indeed, there is a
simple inductive definition of the empire of a node, see
\cite{212898}).  However, the simple proof of the existence of the
empire does not carry over for proof nets with mix.  In addition, the
very notion of ``empire'' is less appealing in the presence of mix.
Without mix, we have that the union of two intersecting subnets is a
subnet, and therefore that the empire of a node $X$ exists if any
subnet with $X$ as a root exists.  Furthermore, we have the following
``simultaneous empire lemma'': if $X$ and $Y$ are two proper nodes,
and $Y$ is not in $e(X)$, then either $e(X) \subset e(Y)$ or $e(X)\cap
e(Y) = \emptyset$.  The following example shows that neither of these
properties hold in the presence of mix:
\begin{example}
\label{ex:badempires}
  Consider the following expansion net, which cannot be derived with
  the \mix rule:
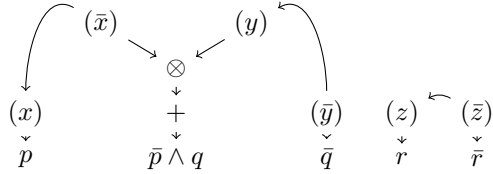
\[
\begin{tikzpicture}[level distance=6mm, grow=up, sibling
    distance=20mm, edge from parent path= {[<-](\tikzparentnode) to
      (\tikzchildnode)}]] 
  \node at (-2,0) (p) {$p$} child {node (x){$(x)$}};
\node at (0,0) (barp) {$\bar p\land q$} child {node{$+$} child {node
    {$\ox$} child {node (y){$(y)$}} child {node (barx){$(\bar x)$}}
            }};
\node at (2,0) (barq) {$\bar q$} child {node (bary){$(\bar y)$}};
\node at (3,0) (r) {$r$} child {node (z) {$(z)$}};
\node at (4,0) (barr) {$\bar r$} child {node (barz){$(\bar z)$}};
\draw[->, in = 90, out = 150] (barx) to (x);
\draw[->, in = 30, out = 90] (bary) to (y);
\draw[->, in = 30, out = 150] (barz) to (z);
\end{tikzpicture}
\]
\noindent The empire of $(\bar x)$ is $(x), (\bar{x}), (z),
(\bar{z})$.  Similarly, the empire of  $(y)$ is $(y), (\bar{y}), (z),
(\bar{z})$.  However, the union of those two subnets is not a subnet,
since there is a switching path from $(\bar{x})$ to $(y)$ outside of it.
Notice also that, while $(\bar x)$ is not in $e((y))$, and  $(y)$ is
not in $e((\bar{x}))$, the two empires intersect (that is, the
simultaneous empire property fails).
\end{example}

In this section we define a more appealing counterpart to the empire
for proof-nets with mix: the ``contiguous empire'' of a node.  It is
easier to show directly that each node has a contiguous empire than to show
directly that each node has a kingdom: moreover, the notion is useful
in proving sequentialization of expansion-nets, and allows us to
define in Section~\ref{sec:basic-cut-reduction} a more pleasing
notion of cut-reduction.

The new notion we introduce, to define the contiguous empire, is the 
property that an AC typed forest is \emph{contiguous} with respect to
one of its roots:
\begin{definition}
\begin{enumerate}
\item Let $F$ be an $AC$ typed forest, and let $X$ be a root of $F$.
  We say that $F$ is \emph{contiguous with respect to $X$} if there is
  a switching path from $X$ to every other node $Y$ of $F$.

\item Let $F$ be an AC typed forest and let $X$ be any proper node of
  $F$.  The \emph{contiguous empire} of $X$ is defined to be the
  largest subnet of $F$ having $X$ as a root which is contiguous with
  respect to $X$.
\end{enumerate}
\end{definition}

\begin{example}
  The expansion net shown in Example~\ref{ex:badempires} is not
  contiguous with respect to any of its roots.  Neither is the empire
  of $(\bar x)$ contiguous with respect to $(\bar x)$.  The
contiguous empire of $(\bar x)$ is $(\bar x), (x)$.
\end{example}

As we will see later, the kingdom of a node is always contiguous, and
so there is no need to consider a concept of ``contiguous kingdom''.
The advantage of the contiguous empire over the usual empire is that
it admits a simple definition, which is a minor variation on the
inductive definition of empires in $ACC$ nets found in~\cite{212898}:

\begin{definition}
  Let $F$ be an $AC$ typed forest and let $X$ be a proper node of $F$. We
  define the substructure $ce(X)$ as the smallest substructure of $F$
  containing $X$ and satisfying the following:
  \begin{itemize}
  \item[$(\ox)$] If $Y = t \ox s$ is a node of $F$, if $t\in ce(X)$ or
    if $s \in ce(X)$, and if and $t,s \neq X$, then $Y$ is in $ce(X)$;
  \item[$(W)$] If $Y = (t \lor *)$ (resp. $(* \lor t)$) is a node of $F$,
    if $t\neq X$, and if $t \in ce(X)$, then $Y$ is
    in $ce(X)$;
  \item [($\Par$ 1)]If $Y$ is a switched node of $F$, and if all the
    successors of $Y$ are in $ce(X)$ and not equal to $X$, then $Y \in
    ce(X)$.
  \item[($\Par$ 2)]If $Y$ is a switched node of $F$, if none of the
    successors of $Y$ are equal to $X$, and if one of the successors
    of $Y$ is in $ce(X)$, then $Y \in ce(X)$ if there is a switching
    path from $X$ to $Y$ which does not pass through any of the
    successors of $Y$ (that is, the path passes into $Y$ from below).
  \end{itemize}
\end{definition}
\begin{remark}
Items $(\ox)$, $(W)$ and $(\Par 1)$ in the above definition are
derived from Girard's inductive definition of the empire in an ACC
net, as described in~\cite{212898}.  This inductive definition fails
in the presence of mix: consider, for example, the following expansion-net (based on an
example from~\cite{Bel97SubMix}):
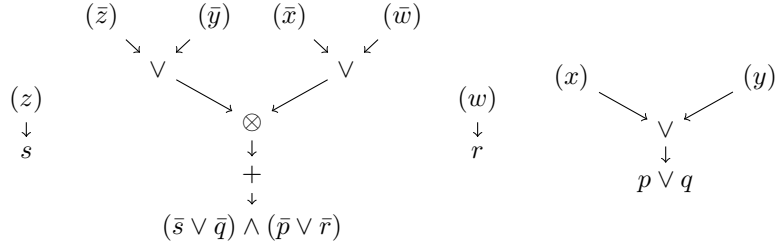
\[
\begin{tikzpicture}[level distance=7mm, grow=up, sibling
    distance=25mm, edge from parent path= {[<-](\tikzparentnode) to
      (\tikzchildnode)}]]
  \node at (0,0) {$(\bar s \lor \bar q)\land( \bar p \lor \bar r)$}
       child {node {$+$} child {node {$\ox$} child {node {$\lor$} 
                            child [sibling
    distance=15mm]{node  {$(\bar{w})$}}
                            child [sibling
    distance=15mm]{node {$(\bar{x})$}}}
                         child {node {$\lor$} 
                            child [sibling
    distance=15mm]{node  {$(\bar{y})$}}
                            child [sibling
    distance=15mm]{node {$(\bar{z})$}}}}};
\node at (-3,1) {$s$}
       child {node {$(z)$}};
\node at (3,1) {$r$}
       child {node {$(w)$}};
\node at (5.5,0.6) {$p \lor q$}
       child {node {$\lor$} child {node{$(y)$}}
                            child {node{$(x)$}}};
\end{tikzpicture}
\]
Applying the (faulty) inductive definition of empire to the rightmost
root $((x) \lor (y))$, we only obtain
the substructure $((x) \lor (y)), (\bar{x}), (\bar{y})$, which is not a
subnet, since there is a switching path through the $\ox$ node from 
$(\bar x)$ to $(\bar{y})$.  However, by using the novel extra
condition $(\Par 2)$, we can observe that since  there is a switching path from
 $((x) \lor (y))$ to $((\bar x) \lor (\bar{w}))$ from below (i.e., via 
$(y)$, $(\bar y)$, and $((\bar{z}) \lor (\bar y))\ox ((\bar x) \lor
(\bar{w}))$),  $((\bar x) \lor (\bar{w}))$ is in the contiguous empire
of $((x) \lor (y))$.  Similarly, $((\bar{z}) \lor (\bar y))$ is in
$ce(((x) \lor (y)))$: from which, applying the other conditions, we
obtain that $ce(((x) \lor (y)))$ is the whole net.
\end{remark}


From the definition of $ce(X)$, we can not immediately see that it is contiguous
with respect to $X$: it is clear that there is a switching path from
$X$ to $Y$ in $F$ for every $Y$ in $ce(X)$, but not clear that
this path lies entirely within $ce(X)$.  The following lemma shows
that $ce(X)$ has an equivalent definition which clearly \emph{is}
contiguous:

\begin{lemma}
Let $F$ be an AC typed forest and $X$ be a proper node of $F$.
 Let $ce'(X)$ be the smallest set of nodes of $F$ containing $\str(X)$
 (the smallest substructure containing $X$) and closed under:
  \begin{itemize}
  \item[$(\ox')$] If $Y = t \ox s$ is a node of $F$, if $t\neq X$ and
    $s \neq X$ and if either $t \in ce'(X)$ or $s\in ce'(X)$, then $Z$
    is in $ce(X)$ for each $Z \in \str(Y)$;
\item[$(W')$] If $Y = (t \lor *)$ (resp. $(* \lor t)$) is a node of $F$,
    $t \neq X$, and either $t \in ce'(X)$, then $Y$ is
    in $ce'(X)$;
\item [($\Par'$ 1)]If $Y$ is a switched node of $F$, and all the
  successors of $Y$ are in $ce'(X)$ and not equal to $X$, then $Y \in
  ce'(X)$.
\item[($\Par'$ 2)]Let $Y$ be a switched node of $F$.  If none of the
  successors of $Y$ are equal to $X$, and if one of the successors of
  $Y$ is $ce'(X)$, then: if there is a switching path $P$ from $X$ to
  $Y$ which does not pass through any of the successors of $Y$, then
  $Z \in ce'(X)$ for each $Z \in \str(W)$, $W \in P$.
  \end{itemize}
\end{lemma}
\begin{proof}
  We must prove that $ce'(X)$ is not larger than $ce(X)$ (it clearly
  contains $ce(X)$).  The difficult case is to show that a structure
  extended by one application of ($\Par'$ 2) can also be extended by
  multiple steps of $(\ox)$, $(W)$,($\Par$ 1), ($\Par$ 2), and closing
  under substructure, as in the definition of $ce(X)$.  We prove this
  by induction on the length of a switching path in the application of
  ( $\Par'$ 2).  Suppose that a single step of ($\Par'$ 2) can be
  carried out by multiple steps of the definition of $ce(X)$ when the
  switching path is of length $<n$. Now suppose ($\Par'$ 2) is applied
  to a structure $G$ and a path $P$ of length $n+1$, ending at a
  switched node $Y$.  By ($\Par$ 2), we may add $Y$ to $G$.  Recall
  that $P$ must enter $Y$ from below.  Counting from $X$, let $W$ be
  the penultimate switched node in $P$ entered from below on $P$ --
  that is, the switching path $Q$ traced from $W$ to $Y$ enters all
  pars in between from a successor). Seen from the opposite direction,
  that means that by applying $(\ox)$ and $(W)$, and closing under
  substructure, we can add $\str(V)$ for every $V$ on the path $Q$
  between $Y$ and $W$.  In particular, at least one of the successors
  of $W$ is a member of $ce(X)$, since $Q$ must leave $W$ by one of
  those successors.  The path $P$ restricted to be from $X$ to $W$
  does not pass through any successor of $W$, and thus by the
  induction hypothesis, we may add the rest of the switching path to
  $ce(X)$.
\end{proof}

\begin{proposition}
$ce(X)$ is contiguous with respect to $X$.
\end{proposition}
\begin{proof}
  By the previous lemma; it is clear that each stage of construction
  of $ce'(X)$ yields a contiguous substructure.
\end{proof}

\begin{proposition}
  Let $F$ be an $AC$-correct structure, and $X$ a proper node of $F$.
  $ce(X)$ is a subnet of $F$.
\end{proposition}

\begin{proof}
  Suppose not.  Then there are roots $Y$, $Z$ of $ce(X)$ such that
  there is a switching path from $Y$ to $Z$ outside of $ce(X)$.  There
  are two cases to consider
  \begin{itemize}
  \item $X$ is $Y$ ($X$ is $Z$ is symmetric).  Then there is a path
    from $X$ to $Z$ inside $ce(X)$, and another outside $ce(X)$.  By
    concatenating these two paths we obtain a cycle, which contradicts
    $AC$-correctness of $F$.

  \item Neither $X$ nor $Y$ is $Z$.  By construction of $ce(X)$, both
    $Y$ and $Z$ are the successors of switched nodes in $F$.  The
    switching path from $Y$ to $Z$ passes through both of those
    switched nodes. In particular, there is a switching path from $Y$
    to $Z'$, the predecessor of $Z$ (a switched node), which enters
    $Z'$ from below.  There is also a path from $X$ to $Y$ within
    $ce(X)$, by contiguousness.  Concatenating these paths, we obtain
    a switching path from $X$ to $Z'$, not via a successor of $Z'$;
    thus $Z'$ is in $ce(X)$, contradicting the fact that $Z$ is a root
    of $ce(X)$.
  \end{itemize}
\end{proof}

\begin{corollary}
Let $F$ be a $AC$-correct structure, and $X$ a proper node of $F$.
  The kingdom $k(X)$ of $X$ exists, and is contiguous
  with respect to $X$.
\end{corollary}
\begin{proof}
  For existence, note that we have demonstrated the existence of a
  subnet $ce(X)$ with $X$ as a root: the kingdom exists by Lemma~\ref{lem:inserseckingdoms}.
  Now consider the subnet $k(X)$ as a net in its own right, with $X$ as a
  root.  By the previous proposition, there is a subnet $ce(X)$ of $k(X)$
  with $X$ as a root; by minimality of the kingdom $ce(X)=k(X)$ and so
  $k(X)$ is contiguous with respect to $X$. 
\end{proof}

We will not need the fact that $ce(X)$ is the contiguous empire of $X$
but we include the proof of that fact here for the sake of completeness.
\begin{proposition}
  $ce(X)$ is the contiguous empire of $X$.
\end{proposition}
\begin{proof}
  Suppose for a contradiction that $G$, some contiguous subnet of $F$,
  contains a node $W_0$ not contained in $ce(X)$.  We may assume that
  this $W_0$ is a switched node $W_0$, which has a successor $Y_0$
  which is a root of $ce(X)$, and a successor $Z_0$ not in $ce(X)$;
  the path from $X$ to any node outside of $ce(X)$ must leave $ce(X)$
  through such a node.  Since $G$ is contiguous, there is a switching
  path from $X$ to $W_0$: since $W_0$ is not a member of $ce(X)$, that
  path must come via $Z_0$, and so $Z_0$ is also a member of $G$.
  Applying the same logic as before, the path from $X$ to $Z_0$ in $G$
  must leave $ce(X)$ at some root $Y_1$ (distinct from $Y_0$ by
  acyclicity).  This root is also the successor of a switched node
  $W_1$, and $W_1$ must also have a successor $Z_1$ not in $ce(X)$.
  Note that we now have $Y_0$, $Y_1$, distinct roots of $ce(X)$,
  successors of switched nodes $W_0, W_1$; those switched nodes each
  have another successor $Z_0$, $Z_1$ not in $ce(X)$. There is a
  switching path from $W_1$ to $W_0$ via $Z_0$, leaving $W_1$ through
  its predecessor.

  Now suppose that, repeating this line of thinking, we have found
  roots $Y_0 \dots Y_n$ of $ce(X)$, successors of switched nodes $W_0,
  \dots W_n$, such that each $W_i$ has another successor $Z_i$ not in
  $ce(X)$, and that there is a switching path from $X$ to each $Y_i$,
  for $i<n$, leaving $ce(X)$ at $Y_{i+1}$.  Suppose, further, that
  there is a switching path from $Z_n$ to $W_1$ which, tracing from
  $W_n$ to $W_1$, enters each $W_i$ via $Z_i$.  Since $W_n$ is in $G$,
  there is a switching path from $X$ to $Z_n$, leaving $ce(X)$ at
  $Y_{n+1}$, which has a predecessor $W_{n+1}$.  There thus a
  switching path from $W_{n+1}$ to $Z_n$, leaving $W_{n+1}$ through
  its predecessor.  By concatenating with the path from $Z_n$ to
  $W_0$, we obtain a switching path from $W_{n+1}$ to $W_1$, and
  consequently to each $W_i$ (to see that this concatenation is really
  a switching path, observe that if there is a switched node common to
  both paths, either there is a switching cycle or a switching path
  from $W_0$ to $W_{n+1}$, contradicting that $W_0$ and $W_{n+1}$ are
  not in $ce(X)$).

  Suppose $Y_{n+1} = Y_j$ for $j\leq n$; then there is a switching
  path from $W_{n+1}$ to itself: a switching cycle.  Thus $Y_{n+1}$ is
  a new root of $ce(X)$.  Since $ce(X)$ has only finitely many nodes,
  eventually $Y_{n+1}$ will be equal to $Y_i$ for some $i$, and we
  obtain a contradiction.
\end{proof}

\subsection{Splitting and sequentialization for expansion nets}
\label{sec:sequ-expans-nets}
Sequentialization for expansion nets is the following:
\begin{theorem}
\label{thm:seq}
  Let $F$ be an e-annotated sequent.  $F$ is an expansion-net
  (i.e. $F$ is derivable in $\LKE$) if and only if $F$ is
  $AC$-correct.
\end{theorem}

The proof of sequentialization for expansion nets is not so different
from sequentialization for \MLLminus plus \mix\ nets.  The
proof of sequentialization  we give in this paper is \emph{bottom-up},
and can be thought of as proof search in $\LKvar$, guided by the
information contained in an e-annotated sequent.
Given an $AC$-correct e-annotated sequent, we look for a
rule of $\LKE$ with $F$ as the conclusion and $AC$-correct e-annotated
sequents as premisses. We call such a root of $F$ a \emph{gate}.
\begin{definition}
  Let $F$ be an $AC$-correct e-annotated sequent.  A \emph{gate} of $F$ is a root $t:A$
  of $F$ which is the conclusion of a rule instance $\rho$ of $\LKE$, such
  that the premisses of $\rho$ are also $AC$-correct e-annotated sequents.
\end{definition}

As we will see below, disjunctions and non-trivial expansions are
always gates.  The major difficulty in proving sequentialization lies
in showing the existence of a gate in the absence of disjunctive and non-trivial
expansions.  In this case, the gate will be a ``splitting'' instance
of  $\ox$.

The proof of the existence of a splitting $\ox$ we give here is
slightly novel, in that we make use of the notion of contiguousness
(the previous proof of this lemma for \mix-nets, by
Bellin~\cite{Bel97SubMix}, is almost the same but less elementary).

\begin{lemma}
  Let $F$ be an $AC$-correct e-annotated sequent, and let the roots of
  $F$ be \emph{trivial} expansions (of the form $(x)$, $(\bar{x})$ or
  $(t \ox s$)).  If at least one root of $F$ is non-atomic, then $F$ has the
  form
  \[ F_1, F_2, ( t \ox s):A \land B,\]
  \noindent where $F_1, t:A$ and $F_2,s: B$ are $AC$-correct e-annotated
  sequents.
\label{lem:splitting}
\end{lemma}
\begin{proof}
Since every root of $F$ is a trivial expansion, we have that
\begin{eqnarray*}
F = (x_1):a_1, \dots (x_n):a_l, (\bar{y}_1):b_1, \dots, (\bar{y}_m):b_m,\\
(t_0 \ox t'_0):A_1\land B_1, \dots (t_n \ox t'_n):A_n\land B_n. 
\end{eqnarray*}
Let $G$ be the typed forest consisting on the witnesses contained in
the roots of $F$: that is,
\begin{eqnarray*}
G = x_1:[a_1], \dots x_n:[a_l], \bar{y}_1:[b_1], \dots, \bar{y}_m:[b_m],\\
t_0 \ox t'_0:A_1\ox B_1, \dots t_n \ox t'_n:A_n\ox B_n. 
\end{eqnarray*}
We will show that there is a root $t\ox s$ of $G$ such that
\[ G = G_1, G_2,  t \ox s:A \ox B,\]
\noindent and such that every path from $G_1$ to $G_2$ in the graph of
$G$ passes through the node $t\ox s$.  In that case, $G_1, t:A$ and
$G_2, s:B$ are $AC$-correct typed forests, and $(t\ox
s):A\land B$ is clearly a gate of $F$.
 
By Proposition~\ref{prop:kingdomordering} $G$ has a $\ll$-maximal root
$X$; this node must be a tensor, without loss of generality $t_0 \ox
t'_0$.  If $X$ is splitting, we are done.  Suppose it is not
splitting; then we know that
  \begin{itemize}
  \item $ce(X)$ is not the whole of $F$, and in addition
  \item there is a (non-switching) path in the graph of $G$ from a root of
    $ce(t_0)$ to a root of $ce(t'_0)$ (if no such path exists, then
    $t_0 \ox t'_0$ is splitting).
  \end{itemize}
  The existence of the path means that there is a root $s$ of
  $ce(t_0)$ whose predecessor is not in $ce(t_0)$, through which this
  path leaves $ce(t_0)$.  This now allows us to discern the existence
  of another $\ll$-maximal node $Y$; if the root $Z$ below $s$ is
  $\ll$-maximal, then we are done; if $Z$ is not $\ll$-maximal, then
  there is an $\ll$-maximal node $Y$ such that $Z \ll Y$.  If $Y$ is
  splitting we are done.

  Now suppose that $t_0\ox t'_0 \dots t_n \ox t'_n$ are all
  $\ll$-maximal non-splitting roots of $G$, and that there is a
  switching path from $t_0$ to $t'_n$ passing through each $t_i$ and
  $t'_i$ in turn.  As above, since $t_n \ox t'_n$ is not splitting
  there is a root $s$ of $ce(t'_i)$ whose predecessor is not in
  $ce(t'_i)$.  This now allows us to discern the existence of another
  $\ll$-maximal node.  If the node is splitting we are done; otherwise
  it is a tensor $t_{n+1} \ox t'_{n+1}$.  Note that it must be
  distinct from the roots of $F$ already listed, otherwise $F$ would
  have a switching cycle: since $k(t_{n+1} \ox t'_{n+1})$ is
  contiguous with respect to $t_{n+1} \ox t'_{n+1}$, there is a
  switching path from, without loss of generality, $t_{n+1}$ to $s$.
  If that switching path enters $s$ from below, then we have a
  switching path from $t_0$ to $Y'_{n+1}$, by concatenation.
  Otherwise, the switching path from $t_{n+1}$ enters $s$ from above,
  and must pass into $ce(Y'_i)$ via some other root $r$; either way,
  we have a switching path from $t_0$ to $t'_{n+1}$.

  We have seen that, given $n$ non-splitting $\ll$-maximal roots of
  $G$, we can find another such root.  Since $G$ has only finitely
  many roots, we eventually find one, $t' \ox s'$, which is
  splitting. So $G = G_1, G_2, t' \ox s'$, where $G_1, t'$, $G_2, s'$
  are both $AC$-correct typed forests. \qed
\end{proof}

\begin{proof}[Of Theorem~\ref{thm:seq}]
By induction on the number of symbols in $F$.  The smallest possible
numbers of symbols in an $AC$-correct annotated sequent is one ($F =
1:\top$), which can easily seen to sequentialize.

Suppose now that $F$ contains more than one symbol.  First, suppose
that $F$ has a graph which is disconnected: let $F'$ be a connected
component of $F$.  Then $F=F', F''$, and we have
\[
\begin{prooftree}
  F' \qquad F'' \justifies F \using \mix
\end{prooftree}
\]
Both $F'$ and $F''$ have
fewer symbols than $F$; by the induction hypothesis there are $\LKE$
derivations of $F'$ and $F''$, and so also an $\LKE$ derivation of
$F$. 

Suppose first that $F$ is an $AC$-correct e-annotated sequent whose graph is
disconnected.  Then each connected component of the graph $F$ defines
an $AC$

We now show that every $AC$-correct e-annotated sequent $F$ whose
graph is connected either has a gate or is of the form $(x):a,
(\bar{x}):\bar{a}$ (i.e. a conclusion of the $\LKE$ axiom) by
induction on the number of nodes in $F$. 

  If $F = F', (t_1 + \cdots + t_n):A$, with $n>2$ then $(t_1 + \cdots
  + t_n):A$ is a gate: for example,$F', (t_1):A, (t_2+ \cdots +
  t_n):A$ is also $AC$-correct, and we have
  \[\begin{prooftree}
    F', (t_1):A, (t_2+ \cdots +
  t_n):A
\justifies
 F', (t_1 + \cdots + t_n):A
\using C
  \end{prooftree}\]
  Similarly, if $F = F', (t\lor s):A\lor B$,
  then $t\lor s$ is a gate: $F', t:A , s:B$ is $AC$-correct and 
 \[\begin{prooftree}
    F', t:A , s:B
\justifies
 F', (t\lor s):A\lor B
\using \lor
  \end{prooftree}
\]
A similar argument shows that roots of the form $(t \lor *)$ and $(*
\lor t)$ are gates.

If $F$ is connected, and no root of $F$ is a disjunction or nontrivial
expansion, then either all the roots of $F$ are of atomic type, or $F$
contains at least one root of conjunctive type.  We may, therefore,
apply Lemma~\ref{lem:splitting} to obtain $AC$-correct e-annotated
sequents $F_1, t':A$ and $F_2, s':B$ such that 
\[
\begin{prooftree}
   F_1, t' \qquad F_2, s'
\justifies
F
\using \land
\end{prooftree}
\]
\noindent is a correct application of the $\land$-rule.  Since the
premisses of this rule are smaller $AC$-correct e-annotated sequent,
they can be derived in $\LKE$, and therefore so can $F$.

Finally, suppose that all the roots of $F$ are trivial expansion of
atomic type, and that $F$ is connected.  Then $F$ must contain at
least one pair $(x):a, (\bar{x}):\bar{a}$, since wire variables occur
in pairs.  It cannot contain any more pairs, since otherwise it would
not be connected, so $F = (x):a, (\bar{x}):\bar{a}$, and is derivable
in $\LKE$.

\end{proof}

\section{Cut-elimination for expansion nets}
\label{sec:cut-elim-expans}
In this final technical section we define a weakly normalizing
cut-elimination procedure directly on expansion-nets which preserves
correctness.  This result did not appear in~\cite{McK10exnets}, and is
new to the current paper. In Propositions
\ref{def:logcutandor},\ref{def:logcutatomic},\ref{def:structcutcontr}
and \ref{def:structcutdefweak}, we show that any individual cut in an
expansion-net can be replaced by ``smaller'' cuts.  Then, in
Lemma~\ref{lem:prinlem}, we show that one cut can be removed
completely from an otherwise cut-free expansion-net.  Finally, in
Theorem~\ref{thm:cut-elim} we show, using the kingdom ordering defined
in the previous section, how to eliminate all cuts from an
expansion-net.

The primary difficulty in defining cut-elimination for classical nets
lies with the reductions involving weakening and contraction.  In the
original linear logic proof nets, deletion and duplication of
subproofs is mediated by boxes.  As we suggested above, in box-free
settings the role of boxes is taken on by \emph{subnets}.  This
means that cut-reduction in classical nets is not \emph{local}: the
subnet to be copied must be calculated, and this calculation can, in
general, consider the whole net.  


Cut-elimination for expansion nets is, of course, strongly
related to cut-elimination for the calculus $\LKvar$. This
calculus is cut-free complete, and so we already have a (semantic)
cut-elimination result, but since this calculus is only complete for
formulae, and not for sequents, the result has a somewhat nonstandard
form:

\begin{proposition}
  Let $\G$ be provable in $\LKvar$ plus $\cut$.  Then there is a 
  sub-multiset $\G'$ of $\G$ provable in $\LKvar$.
\end{proposition}

It is interesting to consider how one might prove this theorem
syntactically, within $\LKvar$.  A typical cut-reduction step in the
sequent calculus is to identify a sub-proof ending with a cut, and
replace it with a sub-proof in which no cuts appear.  Applying that
methodology to $\LKvar$, we take a subproof proving $\G$ and acquire a
subproof proving a subsequent $\G'$.  If we had access to weakening,
this would be unproblematic, but in our setting we can only ``weaken''
within a disjunction.  Thus any formula which ``becomes weak'' during
cut-elimination must be a conclusion of the whole derivation or an
immediate subformula of a disjunction: this is an unusual requirement
for a cut-reduction step; it adds another place in which commutative
conversions must be applied and it is not immediately clear that it
can lead to cut-elimination.  Fortunately, in a proof-net setting
commutative conversions are not needed, and it is easy to see that
such reductions can always be applied.

To begin, we need to introduce a notion of expansion-net with cut:

\begin{definition}
\begin{enumerate}
\item A \emph{cut tree} is an unordered pair of an expansion tree t of
  type $A$, and an expansion tree s of type $\bar{A}$, where $A$ is a
  formula of classical propositional logic not equal to $\top$ or
  $\bot$.  The \emph{positive term} in the cut is the term of type
  $\bar{a}$ or $A \land B$.  We write a cut tree $t \bowtie s$, where
  by convention the positive term is written on the left (when it is
  known which of $s$ and $t$ is the positive term).

\item A typed forest with cut is a finite forest of typed
  propositional expansion trees, witnesses and cut trees, such that a
  wire-variable $x$ occurs at most once, and occurs if and only if its
  dual $\bar{x}$ occurs.  The type of a typed forest with cut is the
  forest of types of its roots which are not cuts.

\item An \emph{expansion-net with cut} is a typed forest with cut,
  derivable in $\LKE$ plus the rule
\[
\begin{prooftree}
F, \ t:A \qquad G, \ s:\bar{A}
\justifies
F, \ G, \ t\bowtie s
\using \Cut
\end{prooftree}
\] 
\end{enumerate}
\end{definition}

Extending the correctness criterion to cuts is trivial, as usual in
proof-nets: we simply treat the cut $t \bowtie s$ as though it were a
conjunctive witness $t \ox s$.  (i.e. an unswitched binary node).  The
notions of subnet, kingdom etc. carry over in an obvious manner, as
does the sequentialization theorem.

\begin{example}
\label{ex:runningex}
The following is a correct expansion-net with cuts: it is derived by
cutting the expansion net in Example~\ref{ex:forestexample} with the expansion-net
witnessing the associativity of $\land$:
\[
    \begin{tikzpicture}[level distance=6.5mm, grow=up, sibling
    distance=10mm, edge from parent path= {[<-](\tikzparentnode) to
      (\tikzchildnode)}]]
  \node at (-1, 0) {$+$} 
    child [sibling
    distance=10mm]{node {$\bar{x}$}} 
    child [sibling
    distance=10mm] {node{$\bar{y}$}}
    child [sibling
    distance=10mm] {node{$\bar{z}$}};
  \node at (5,0) {$\bowtie$}
     child [sibling
    distance=30mm]{node {$\lor$} child [sibling
    distance=10mm]{node {$(\bar l)$}}
                                 child [sibling
    distance=10mm]{node (or) {$\lor$}
        child [sibling
    distance=10mm] {node(barm){$(\bar{m})$}}
        child [sibling
    distance=10mm] {node{$+$} child {node {$\ox$}
             child {node (barn){$(\bar{n})$}}
             child {node (baro){$(\bar{o})$}}}}}}
    child [sibling
    distance=30mm] {node {$+$} 
           child [sibling
    distance=15mm] {node{$\ox$} child {node {$(z)$}}
                                child {node {$+$} 
           child [sibling
    distance=15mm] {node{$\ox$} 
                               child [sibling
    distance=10mm]{node{$(w)$}}
                               child [sibling
    distance=10mm] {node {$\lor$} child {node{$*$}} child
                                 {node{$(y)$}}}}
                       child [sibling
    distance=15mm]{node{$\ox$} 
                               child [sibling
    distance=10mm] {node{$(v)$}}
                               child [sibling
    distance=10mm] {node {$\lor$} child {node{$(x)$}}
                                  child {node{$*$}}}}}}};
    \node at (1,0) {$+$}     child {node {$\ox$} 
      child {node{$(\bar{v})$}}
      child {node{$(\bar{w})$}}};
  \node at (9,0) {$+$}
        child {node {$\ox$}
                child {node {$+$}
                   child {node {$\ox$}
                           child {node {$(l)$}}
                           child {node {$(m)$}}}}
                child {node {$\lor$}
                   child {node{$(n)$}}
                   child {node{$(o)$}}}};
\end{tikzpicture}
\]
\end{example}

\subsection{The basic cut-reduction operations}
\label{sec:basic-cut-reduction}
As in Gentzen-style cut-elimination for sequent calculi,
cut-elimination in expansion nets is based on a number of basic
operations.  In general, the application of these rules may not
terminate, but we can find a \emph{strategy} for applying these rules
such that there is a measure on proofs which decreases. 
We define cut-reduction, not just on expansion nets, but on $AC$ typed
forests which might, in general, have witnesses as roots. This allows
us to perform cut-reduction on subnets of an expansion-net, just as
one eliminates cuts in a subproof in the sequent calculus. 

Unlike usual cut-elimination results, the cut-elimination we define in
this section does not in general preserve the type of derivations.
This is for two reasons.  The first has been mentioned above: namely
that expansion-nets with cut are sequent-complete (can prove all
sequents derivable in \LK) while expansion-nets without cut are only
formula complete.  The second reason concerns cut-elimination in a
general $AC$ typed forest.  Consider the following example:

\[ 
\begin{tikzpicture}[level distance=6.5mm, grow=up, sibling
    distance=10mm, edge from parent path= {[<-](\tikzparentnode) to
      (\tikzchildnode)}]]
  \node at (-2, 0) {$[p]$} child {node(x){$x$}};
\node at (0, 0) {$\bowtie$} child {node{$+$} child {node(z){$z$}} child {node(y){$y$}} }  
                            child {node{$+$} child {node(barx){$\bar x$}} };
\node at (2.5,0) {$\bar p \land \bar p$} child {node {$\ox$}
                  child {node (barz){$(\bar z)$}}
                  child {node (bary){$(\bar y)$}}};
\draw [->, in=50, out = 150] (barx) to (x);
\draw [->, in=30, out = 150] (bary) to (y);
\draw [->, in=30, out = 150] (barz) to (z);
\end{tikzpicture}
\]

\noindent Cuts of this form are reduced by ``yanking'': the axiom link
between the $x$ and $\bar x$, and the cut, disappear, leaving the
following net:
 
\[
\begin{tikzpicture}[level distance=6.5mm, grow=up, sibling
    distance=10mm, edge from parent path= {[<-](\tikzparentnode) to
      (\tikzchildnode)}]]
 
\node at (0, 0) {$p$} child {node {$+$} child {node{$z$}} child {node{$y$}}} ;  
                     
\node at (2.5,0) {$\bar p \land \bar p$} child {node {$\ox$}
                  child {node {$(\bar z)$}}
                  child {node {$(\bar y)$}}};
\end{tikzpicture}
\] 

However, the type of the net has changed: we have replaced a root of
type $[p]$ with a root of type $p$. We will use the term
$\emph{closure}$ to describe a sequent resulting from deleting some
formulae, and replacing others with their underlying type:

\begin{definition}
  Let $\G$ be a forest of types. A \emph{closure} of $\G$ is a
  forest $\G'$ of types, together with an injective function from 
  the roots of $\G'$ to the roots of $G$ which either preserves types
  or replaces a type with its underlying classical formula (see Definition~\ref{def:types}).
\end{definition}

\noindent Notice that if $\G$ does not contain any witness types, a
closure of $\G$ is just a sub-multiset of $\G$.

Our cut-elimination argument relies on isolating, in an expansion net
$F$, a subnet $G$ containing only one cut, and replacing it with a
cut-free AC typed forest $G'$ whose type is a closure of $G$: that is,
we replace each non-cut root $t$ of $G$ with $f^{-1}(G)$, where $f$
given to us by the closure of the type of $G$.  If $t$ has no
pre-image, and is a root, we can delete it.  If $t$ has no pre-image,
and is a root, we would like to replace it by $*$ (representing that
the formula previously introduced by $t$ is now introduced by
weakening) -- however, we must be careful to ensure that the $*$
occurs within a default weakening. For example, consider the following
expansion-net, with marked subnet:

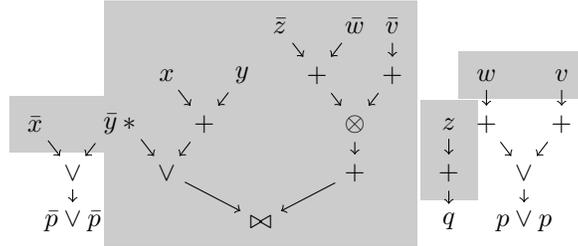
\[
\begin{tikzpicture}[level distance=6.5mm, grow=up, sibling
    distance=10mm, edge from parent path= {[<-](\tikzparentnode) to
      (\tikzchildnode)}]]
    \node at (-2.5, 0) {$\bar{p}\lor \bar p$} child {node  {$\lor$}
      child {node(bary){$\bar y$}} 
child {node(barx){$\bar x$}}};
\node at (0, 0) (cut) {$\bowtie$} child[sibling distance = 25mm] {node{$+$} child[sibling distance = 10mm]  {node {$\ox$} child
    {node{$+$} child {node(barv){$\bar v$}}}
                                             child {node{$+$} child {node {$\bar w$}} child {node
                                  {$\bar z$}}}}}
                            child[sibling distance = 25mm]  {node {$\lor$} child[sibling distance = 10mm]  {node {$+$}
                                child {node{$y$}} child {node{$x$}}}
                                                 child [sibling distance = 10mm] {node (star){$*$}}};
\node at (2.5,0)  {$q$} child {node (plus){$+$} child {node(z) {$z$}}};
\node at (3.5,0) {$p \lor p$} child {node {$\lor$} child {node {$+$} child {node (v){$v$}}}
                          child {node {$+$} child {node (w){$w$}}}};
 \begin{pgfonlayer}{background}
       \node [fill=black!20,fit=(cut) (star) (barv) ] {};
 \node [fill=black!20,fit=(barx) (bary)] {};
\node [fill=black!20,fit=(w)(v)] {};
\node [fill=black!20,fit=(plus)(z)] {};
     \end{pgfonlayer}
\end{tikzpicture}
\]
\noindent The marked subnet $G$ has type $[\bar p], [\bar p], q, [p], [p]$.  The AC typed forest
$G' = \bar{x}:[p], \bar{y}:[q], (x+y): p$ has type which is a completion of
the marked subnet, via an injection $f$ which does not have a preimage
for $q$ or the first copy of $[p]$.  The trees missing a pre-image are
$(z)$ (which is a root) and $w$ (which is in a disjunction), so the
result of replacing $G$ by $G'$ is an expansion net:
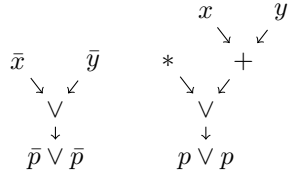
\[
\begin{tikzpicture}[level distance=6.5mm, grow=up, sibling
    distance=10mm, edge from parent path= {[<-](\tikzparentnode) to
      (\tikzchildnode)}]]
    \node at (0, 0) {$\bar{p}\lor \bar p$} child {node  {$\lor$}
      child {node(bary){$\bar y$}} 
child {node(barx){$\bar x$}}};

\node at (2,0) {$p \lor p$} child {node {$\lor$} child {node {$+$} child {node (v){$y$}}child {node (v){$x$}}}
                          child {node {$*$}}};
\end{tikzpicture}
\]

The following proposition shows that we can replace a subnet $G$ by
another $AC$ typed forest whose type is a closure of $G$, provided the
deleted roots fall inside disjunctions or expansions:
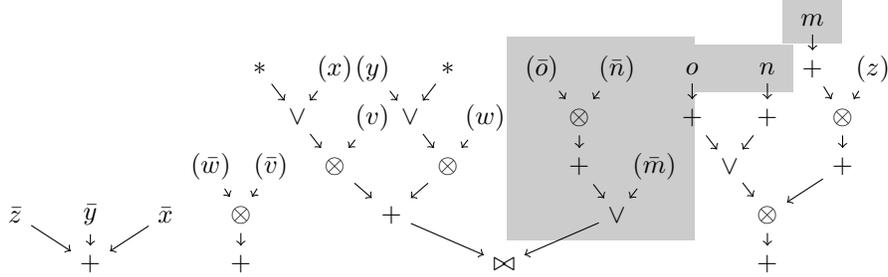
\begin{proposition}
\label{prop:subnetfornetcomp}
\item Let $G$ be a subnet of an $AC$ typed forest $F$, and let $G'$ be
  an $AC$ typed forest whose type is a closure $G$: that is, there is
  an injective function $f$ from the non-cut roots of $G'$ to those of
  $G$ such that $f$ either preserves types or maps a term of witness
  type to a term of its underlying type. Call a node $t$ of $F$ weak
  if it is a root of $G$ but has no $f$-preimage. Suppose further that
  each weak node of $F$ is either a root of $F$, or the sucessor of a
  switched node $Y$ (a disjunction or expansion) such that at least
  one other successor of $Y$ is not weak.  Then, replacing $G$ by $G'$
  in $F$ (by replacing $t$ by $f^{-1}(t)$ if $t$ has an $f$-preimage,
  deleting $t$ if it is weak and a root or a successor of an
  expansion, and replacing $t$ by $*$ if it is weak and the successor
  of a disjunctive node, and replacing the cuts of $G$ by the cuts of
  $G'$) yields an $AC$ typed forest whose type is a closure of the
  type of $F$.
\end{proposition}

\begin{proof}
  An easy examination of the correctness criterion.
\end{proof}

We now introduce the basic reductions of expansion-nets, and show that
they preserve AC correctness.  We illustrate the reductions by showing
how to reduce the net in Example~\ref{ex:runningex} to cut-free form.

We will define a \emph{logical cut} to be one in which the positive
cut term is an expansion consisting of a single witness.  In case the 
cut has non-atomic type, the definition of cut-reduction is easy: 
\begin{proposition}[Logical cut -- $\land/\lor$]
\label{def:logcutandor}
Let $G = F, (s_1 \ox s_2)\bowtie (t_1 \lor t_2)$ be an AC typed
forest, such that $t_1, t_2 \neq *$.  Then $G' = F, s_1 \bowtie t_1,
s_2 \bowtie t_2 $ is an AC typed forest with the same type as $G$.
\end{proposition}

Reducing the logical cut in Example~\ref{ex:runningex}, we obtain the
following net:
\[
    \begin{tikzpicture}[level distance=6.5mm, grow=up, sibling
    distance=8mm, edge from parent path= {[<-](\tikzparentnode) to
      (\tikzchildnode)}]]
  \node at (-1, 0) {$+$} 
    child [sibling
    distance=10mm]{node {$\bar{x}$}} 
    child [sibling
    distance=10mm] {node{$\bar{y}$}}
    child [sibling
    distance=10mm] {node{$\bar{z}$}};
\node at (7.5,0) {$\bowtie$} child {node {$+$} child {node {$\bar l$}}}
                           child {node {$+$} child {node {$z$}}};
  \node at (5,0) {$\bowtie$}
                                 child [sibling
    distance=20mm]{node (or) {$\lor$}
        child [sibling
    distance=10mm] {node(barm){$(\bar{m})$}}
        child [sibling
    distance=10mm] {node{$+$} child {node {$\ox$}
             child {node (barn){$(\bar{n})$}}
             child {node (baro){$(\bar{o})$}}}}}
    child [sibling
    distance=30mm] {node {$+$} 
           child [sibling
    distance=15mm] {node{$\ox$} 
                               child [sibling
    distance=10mm]{node{$(w)$}}
                               child [sibling
    distance=10mm] {node {$\lor$} child {node{$*$}} child
                                 {node{$(y)$}}}}
                       child [sibling
    distance=15mm]{node{$\ox$} 
                               child [sibling
    distance=10mm] {node{$(v)$}}
                               child [sibling
    distance=10mm] {node {$\lor$} child {node{$(x)$}}
                                  child {node{$*$}}}}};
    \node at (1,0) {$+$}     child {node {$\ox$} 
      child {node{$(\bar{v})$}}
      child {node{$(\bar{w})$}}};
  \node at (9,0) {$+$}
        child {node {$\ox$}
                child {node {$+$}
                   child {node {$\ox$}
                           child {node {$(l)$}}
                           child {node {$(m)$}}}}
                child {node {$\lor$}
                   child {node{$(n)$}}
                   child {node{$(o)$}}}};
\end{tikzpicture}
\]

As discussed above, in case of an atomic logical cut, the whole forest has the form $G = F, (x)\bowtie
t$, where $t$ is a possibly nontrivial sum of witnesses.  We want to
reduce this cut, as in usual proof-nets, by ``yanking'':

\begin{proposition}[Logical cut -- atomic]
\label{def:logcutatomic}
Let $G = F, (x)\bowtie t$ be an AC typed forest.  Then $F[\bar{x} :=
t]$ is an AC typed forest with type $G'$ a closure of $G$.
\end{proposition}

\noindent Reducing the atomic logical cut in our example, we obtain:

\[
    \begin{tikzpicture}[level distance=6.5mm, grow=up, sibling
    distance=8mm, edge from parent path= {[<-](\tikzparentnode) to
      (\tikzchildnode)}]]
  \node at (-1, 0) {$+$} 
    child [sibling
    distance=10mm]{node {$\bar{x}$}} 
    child [sibling
    distance=10mm] {node{$\bar{y}$}}
    child [sibling
    distance=10mm] {node{$\bar{z}$}};
  \node at (4.5,0) {$\bowtie$}
                                 child [sibling
    distance=30mm]{node (or) {$\lor$}
        child [sibling
    distance=10mm] {node(barm){$(\bar{m})$}}
        child [sibling
    distance=10mm] {node{$+$} child {node {$\ox$}
             child {node {$(\bar{n})$}}
             child {node (baro){$(\bar{o})$}}}}}
    child [sibling
    distance=30mm] {node {$+$} 
           child [sibling
    distance=15mm] {node{$\ox$} 
                               child [sibling
    distance=10mm]{node{$(w)$}}
                               child [sibling
    distance=10mm] {node {$\lor$} child {node{$*$}} child
                                 {node{$(y)$}}}}
                       child [sibling
    distance=15mm]{node{$\ox$} 
                               child [sibling
    distance=10mm] {node{$(v)$}}
                               child [sibling
    distance=10mm] {node {$\lor$} child {node{$(x)$}}
                                  child {node{$*$}}}}};
    \node at (1,0) {$+$}     child {node {$\ox$} 
      child {node{$(\bar{v})$}}
      child {node{$(\bar{w})$}}};
  \node at (8,0) {$+$}
        child {node {$\ox$}
                child [sibling
    distance=20mm]{node {$+$}
                   child [sibling
    distance=8mm] {node {$\ox$}
                           child {node {$(z)$}}
                           child {node {$+$} child {node(m){$m$}}}}}
                child [sibling
    distance=10mm]{node {$\lor$}
                   child [sibling
    distance=10mm]{node {$+$} child{ node(n){$n$}}}
                   child {node {$+$} child {node(o){$o$}}}}};
 \begin{pgfonlayer}{background}
       \node [fill=black!20,fit=(or) (barm) (barn) (baro)] {};
 \node [fill=black!20,fit=(o) (n)] {};
\node [fill=black!20,fit=(m)] {};
     \end{pgfonlayer}
\end{tikzpicture}
\]

Cut against contraction is dealt with by reducing the \emph{positive
  width} of the cut: the number of witnesses appearing in the positive
cut term.  This is achieved by duplicating a subnet with the positive
cut term as a root (the equivalent of duplicating a subproof with the
positive cut-formula in the conclusion).  In duplicating a subnet we
must make sure to rename the wire-variables so that no variable occurs
more than once: given a term $t$, we use the notation $t^L$, $t^R$ to
denote two copies of $t$ where each wire variable $x$ has been
replaced by fresh variables $x^L$, $x^R$, and each wire variable
$\bar{x}$ has been replaced by fresh variables $\bar{x}^L$,
$\bar{x}^R$, such that $\bar{x}^L$ is dual to $x^L$, and so on.

\begin{proposition}[Structural cut -- contraction]
\label{def:structcutcontr}
Let 
\[G = F, (s_1+ \dots + s_n) \bowtie_X t\] be an AC typed forest, where
$s = (s_1+ \dots + s_n)$ is nontrivial expansion.  Let $s= s_1 + s_2$
(that is, $s_1$ and $s_2$ partition $s$) and let $w_1\dots, w_n, c_1,
\dots c_m, t$ be a subnet of $G$, whose roots other than $t$ are all
either witnesses (the $w_i$), or cuts (the $c_i$).  Let $F'$ be the
forest defined by replacing each $w_i$ by $(w^L_i+w^R_i)$.  Then \[G'
= F', s_1 \bowtie t^L, s_2 \bowtie t^R\] is an AC typed forest, and
the type of $G'$ is a closure of the type of $G$.
\end{proposition}
\begin{proof}
  Let $H,(s_1+s_2)$ be the kingdom of $(s_1 + s_2)$: then
  \[J = H, w_1, \dots, w_n, c_1, \dots c_m, (s_1+s_2)\bowtie t\] is a
  subnet of $G$.  It is easy to see that \[J' = H, (w^L_1+w^R_1),
  \dots, (w^L_n+w^R_n), c^L_1, c^R_1 \dots c^L_m, c^R_m, s_1\bowtie
  t^L, s_2\bowtie t^R\] is an $AC$ typed forest with type a closure of
  the type of $J$.  Since $J$ was a subnet, the result of replacing $J$
  in $F$ by $J'$ also $AC$; the result follows.
\end{proof}

\begin{remark}
  The restriction that the duplicated subnet have only witnesses and
  cuts as roots ensures that we can ``contract'' the duplicated
  conclusions by adding expansions.  The \emph{kingdom} of the positive
  cut-term always yields such a subnet; if a root $s$ of the kingdom
  of $t$ were a disjunction or expansion, we could remove that node to
  yield a smaller subnet with $t$ as a root.
\end{remark}

In the last step of our running example, the kingdom of the negative
branch of the cut is shaded.  Notice that, apart from the root taking
part in the cut, all the roots of the kingdom are witnesses.  Thus, we
can duplicate the kingdom, cutting each copy against one of the
witnesses on the positive branch of the cut:
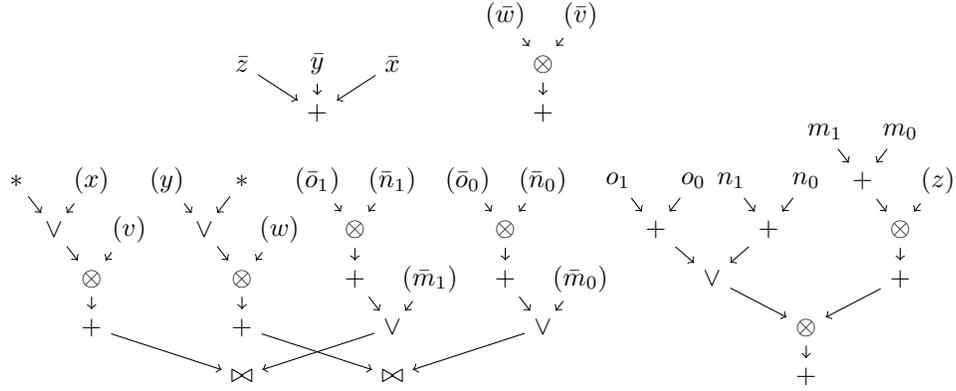
\[
    \begin{tikzpicture}[level distance=6.5mm, grow=up, sibling
    distance=10mm, edge from parent path= {[<-](\tikzparentnode) to
      (\tikzchildnode)}]]
  \node at (0, 3.5) {$+$} 
    child [sibling
    distance=10mm]{node {$\bar{x}$}} 
    child [sibling
    distance=10mm] {node{$\bar{y}$}}
    child [sibling
    distance=10mm] {node{$\bar{z}$}};
  \node at (1,0) {$\bowtie$}
                                 child [sibling
    distance=40mm]{node (or) {$\lor$}
        child [sibling
    distance=10mm] {node(barm){$(\bar{m}_0)$}}
        child [sibling
    distance=10mm] {node{$+$} child {node {$\ox$}
             child {node {$(\bar{n}_0)$}}
             child {node (baro){$(\bar{o}_0)$}}}}}
    child [sibling
    distance=40mm] {node {$+$} 
           child [sibling
    distance=15mm] {node{$\ox$} 
                               child [sibling
    distance=10mm]{node{$(w)$}}
                               child [sibling
    distance=10mm] {node {$\lor$} child {node{$*$}} child
                                 {node{$(y)$}}}}};
  \node at (-1,0) {$\bowtie$}
                                 child [sibling
    distance=40mm]{node (or) {$\lor$}
        child [sibling
    distance=10mm] {node(barm){$(\bar{m}_1)$}}
        child [sibling
    distance=10mm] {node{$+$} child {node {$\ox$}
             child {node {$(\bar{n}_1)$}}
             child {node (baro){$(\bar{o}_1)$}}}}}
                       child [sibling
    distance=40mm]{node {$+$} child {node{$\ox$} 
                               child [sibling
    distance=10mm] {node{$(v)$}}
                               child [sibling
    distance=10mm] {node {$\lor$} child {node{$(x)$}}
                                  child {node{$*$}}}}};
    \node at (3,3.5) {$+$}     child {node {$\ox$} 
      child {node{$(\bar{v})$}}
      child {node{$(\bar{w})$}}};
  \node at (6.5,0) {$+$}
        child {node {$\ox$}
                child [sibling
    distance=25mm]{node {$+$}
                   child [sibling
    distance=10mm] {node {$\ox$}
                           child {node {$(z)$}}
                           child [sibling
    distance=10mm] {node {$+$} child {node(m){$m_0$}}child {node(m){$m_1$}}}}}
                child [sibling
    distance=25mm]{node {$\lor$}
                   child [sibling
    distance=15mm]{node {$+$} child [sibling
    distance=10mm]{ node(n){$n_0$}}child [sibling
    distance=10mm]{node(m){$n_1$}}}
                   child [sibling
    distance=15mm] {node {$+$} child [sibling
    distance=10mm] {node(o){$o_0$}} child [sibling
    distance=10mm]{node(m){$o_1$}}}}};
\end{tikzpicture}
\]

After some logical cuts, we arrive at the following net:
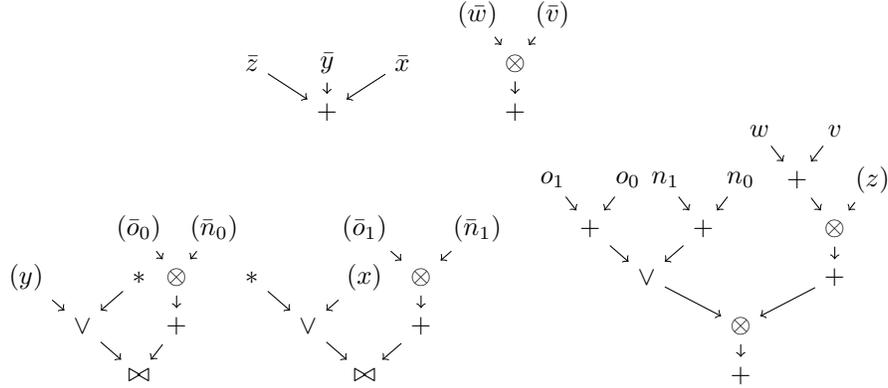
\[
    \begin{tikzpicture}[level distance=6.5mm, grow=up, sibling
    distance=10mm, edge from parent path= {[<-](\tikzparentnode) to
      (\tikzchildnode)}]]
  \node at (3.5, 3.5) {$+$} 
    child [sibling
    distance=10mm]{node {$\bar{x}$}} 
    child [sibling
    distance=10mm] {node{$\bar{y}$}}
    child [sibling
    distance=10mm] {node{$\bar{z}$}};
  \node at (1,0) {$\bowtie$}
  child {node{$+$} child {node {$\ox$}
             child {node {$(\bar{n}_0)$}}
             child {node (baro){$(\bar{o}_0)$}}}}
    child [sibling
    distance=15mm]{node {$\lor$} child {node{$*$}} child
                                 {node{$(y)$}}};
  \node at (4,0) {$\bowtie$}
                                 child [sibling
    distance=15mm] {node{$+$} child {node {$\ox$}
             child {node {$(\bar{n}_1)$}}
             child {node (baro){$(\bar{o}_1)$}}}}
                       child [sibling
    distance=15mm]{node {$\lor$} child {node{$(x)$}}
                                  child {node{$*$}}};
    \node at (6,3.5) {$+$}     child {node {$\ox$} 
      child {node{$(\bar{v})$}}
      child {node{$(\bar{w})$}}};
  \node at (9,0) {$+$}
        child {node {$\ox$}
                child [sibling
    distance=25mm]{node {$+$}
                   child [sibling
    distance=10mm] {node {$\ox$}
                           child {node {$(z)$}}
                           child [sibling
    distance=10mm] {node {$+$} child {node(m){$v$}}child {node(m){$w$}}}}}
                child [sibling
    distance=25mm]{node {$\lor$}
                   child [sibling
    distance=15mm]{node {$+$} child [sibling
    distance=10mm]{ node(n){$n_0$}}child [sibling
    distance=10mm]{node(m){$n_1$}}}
                   child [sibling
    distance=15mm] {node {$+$} child [sibling
    distance=10mm] {node(o){$o_0$}} child [sibling
    distance=10mm]{node(m){$o_1$}}}}};
\end{tikzpicture}
\]
This net contains two examples of our final kind of cut: a cut against
default weakening.  This situation superficially resembles the
a cut between the \emph{additive} propositional connectives in sequent
calculus.  In common with the reduction for such a cut, we delete a subproof (here
subnet) of the proof.  Unlike the additive reduction, we must replace
the conclusions of the deleted subnet by weakenings: the catch here is
that, to ensure that each weakening thus created is a default
weakening, each weakened subtree must either be a component of a
nontrivial expansion or of a disjunctive term which is not already a
default weakening.
\begin{proposition}[Structural cut -- default weakening]
\label{def:structcutdefweak}
Let \[G = F, (s_1 \ox s_2)\bowtie (t \lor *)\] be an AC typed forest.
Let $E = u_1, \dots, u_n, s_2$ be a subnet of $G$, such that each tree
$u_i$ is either a root of $G$, a successor of an expansion containing
at least one term not in $E$, or is the successor of a disjunction
node the other successor of which is neither an instance of $*$ nor in
$E$.  Let $F'$ be the forest derived from $F$ as follows: if $u_i$ is
a root of $F$, delete it: otherwise, replace it by $*$.  Then
\[G' = F', s_1 \bowtie 
t\] 
is a default attached AC forest, and the type of $G'$ is a
subsequent of the type of $G$.
\end{proposition}
\begin{proof}
  Let $L, s_1$ and $M, t$ be the kingdoms of $s_1$ and $t$
  respectively.  Then \[L, M, u_1, \dots, u_n, (s_1 \ox s_2)\bowtie (t
  \lor *)\] is a subnet of $G$, and $L, M, s_2 \bowtie t$ is an $AC$
  forest.  The forest $F', s_1 \bowtie t$ is therefore $AC$ correct:
  it is default-attached, since every non-root term replaced by $*$ is
  either in a non-trivial expansion or forms an attached weakening.
\end{proof}

\begin{remark}
  There are two nets we can canonically choose to delete which satisfy
  the conditions on $E$ above: namely the empire and the contiguous
  empire of $s_2$.  This follows immediately from the definitions of
  (contiguous) empire.  Our strategy for cut-elimination will always
  delete the contiguous empire, for the following reason. Consider the
  following rule instance in $\LKE$:
\[
\begin{prooftree}
F, (s_1 \ox s_2)\bowtie (t \lor *) \qquad G
\justifies
F, G, (s_1 \ox s_2)\bowtie (t \lor *)
\using \mix
\end{prooftree}
\]
The empire of $s_2$ changes after application of the rule, while the
contiguous empire stays the same.  If we delete the empire of $s_2$,
then it matters in which subproof we perform the reduction, while
deleting the contiguous empire is independent of that choice. Thus,
deleting $ce(s_2)$ is more \emph{compositional} than deleting
$e(s_2)$, since the result depends less on the context in which the
reduction takes place.

The reduction thus defined is not, however, entirely compositional: If
$u:A$ is in the contiguous empire of $s_2$ in $F, u:A, (s_1 \ox
s_2)\bowtie (t \lor *)$, then before cut reduction we can form a
conjunction on $A$, and afterwards we cannot.  This problem is,
however, not so drastic; if instead we reduce the cut after
introducing the conjunction, in addition to what was deleted before,
we also delete the contiguous empire of the other conjunct, which
becomes part of the contiguous empire of $s_2$.  In other words, the
part of the proof which could not be introduced via \mix\ will in any
case be deleted after cut-reduction.
\end{remark}

In our running example, the contiguous empire of $(\bar{n}_0)$ is the 
forest $(\bar{n}_0), n_0$, and the  contiguous empire of $(\bar{o}_1)$ is the 
forest $(\bar{o}_1), o_1$.  The resulting cut-free net, after the
structural reductions and a number of logical reductions, is

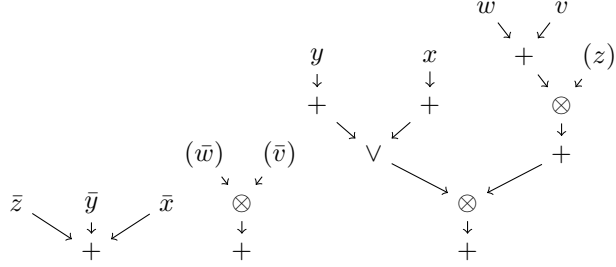
\[
    \begin{tikzpicture}[level distance=6.5mm, grow=up, sibling
    distance=10mm, edge from parent path= {[<-](\tikzparentnode) to
      (\tikzchildnode)}]]
  \node at (-4, 0) {$+$} 
    child [sibling
    distance=10mm]{node {$\bar{x}$}} 
    child [sibling
    distance=10mm] {node{$\bar{y}$}}
    child [sibling
    distance=10mm] {node{$\bar{z}$}};

    \node at (-2,0) {$+$}     child {node {$\ox$} 
      child {node{$(\bar{v})$}}
      child {node{$(\bar{w})$}}};
  \node at (1,0) {$+$}
        child {node {$\ox$}
                child [sibling
    distance=25mm]{node {$+$}
                   child [sibling
    distance=10mm] {node {$\ox$}
                           child {node {$(z)$}}
                           child [sibling
    distance=10mm] {node {$+$} child {node(m){$v$}}child {node(m){$w$}}}}}
                child [sibling
    distance=25mm]{node {$\lor$}
                   child [sibling
    distance=15mm]{node {$+$} child [sibling
    distance=10mm]{ node(n){$x$}}}
                   child [sibling
    distance=15mm] {node {$+$} child [sibling
    distance=10mm] {node(o){$y$}}}}};
\end{tikzpicture}
\]

\subsection{Cut-elimination theorem for expansion nets}
\label{sec:cut-elim-theor}
The core of cut-elimination is the following lemma, which states that
a single ``topmost'' cut can be removed from an expansion-net. Topmost
is here defined using the relation $\ll$: given two cuts $X$ and $Y$,
if $X \ll Y$ then $X$ is in the kingdom of $Y$: thus, a cut $Z$ which
is minimal among the cuts of $F$ with respect to $\ll$ is not in the
kingdom of any other cut, and so there is at least one
sequentialization of the net such that the proof above $Z$ is
cut-free. Furthermore, the lemma states that this topmost cut can be removed
in such a way that duplications happen only within the kingdom of the
cut: that is, the cut is eliminated by replacing the kingdom of the
cut with a cut free $AC$ forest, plus some supplementary deletions.
\begin{lemma}[Principal lemma for default-attached nets]
\label{lem:prinlem}
  Let $G = F, t\bowtie_X s$ be an AC forest containing $n+1$ cuts, and
  let the cut $\bowtie_X$ be $\ll$-minimal among the cuts in $G$. By
  applying the transformations in Propositions~\ref{def:logcutatomic},
  \ref{def:logcutandor},~\ref{def:structcutcontr} and
  \ref{def:structcutdefweak} to $G$, we can obtain an $AC$ forest
  $G'$, containing $n$ cuts, such that
  \begin{enumerate}
  \item $G$ and $G'$ only differ on the part of $G$ disjoint from the
    contiguous empire of $X$ in $G$.
  \item Outside of the kingdom of $X$ in $G$, $G$ and $G'$ only differ
    by the deletion of subtrees or their replacement with $*$.
    \end{enumerate}
The type of $G'$ is a closure of the type of $G$.
\end{lemma}

\begin{proof}
  By induction on the rank of the cut-formula, with a sub-induction on
  the positive width of the cut.  Suppose first that the cut-formula
  is atomic, and that the cut has the form $x \bowtie_X t$.  The
  kingdom of $X$ is $H, \bar{x}, x \bowtie_X t$, and the atomic
  cut reduction replaces this subnet by $H, t$. Nothing outside the
  kingdom of $X$ is changed.

  Now assume, as an induction hypothesis, that the lemma holds for a
  $\ll$-minimal cut of rank $n$ and positive width $m$.  Suppose first
  that the cut has the form $(t \lor *) \bowtie (s_1 \otimes s_2)$.
  To reduce this cut, we delete $ce(s_2)$, the contiguous empire
  of $s_2$.  After one step of reduction, we obtain an $AC$ forest $F'
  = G', t \bowtie_Y s_1$; the roots of $k(Y)$/$ce(Y)$ are contained
  within the roots of $k(X)$/$ce(X)$.  Apply the induction hypothesis
  to $F'$ to obtain a cut-free expansion net with the required
  properties.

  Now, suppose that the cut has the form $(t_1 \lor t_2) \bowtie_X
  (s_1 \otimes s_2)$.  After one step of cut-reduction, we obtain the
  $AC$ forest $F, t_1 \bowtie_Y s_1, t_2\bowtie_Z s_2$.  Note that $Z
  \notin k(Y)$ and $Y \notin k(Z)$, and so both $Z$ and $Y$ are
  $\ll$-minimal; also note that if $u$ is a root of $k(Y)$ not equal
  to $Y$ (or of $k(Z)$ not equal to $Z$), then $u$ is contained in a
  root of $k(X)$.  Apply the induction hypothesis to one of the cuts,
  without loss of generality $Y$.  The important thing to note is that, since $Z$ is not in
  $k(Y)$, it is not duplicated by eliminating $Y$, though it may be
  deleted, since it is in $ce(Y)$.  If it is deleted, we are done:
  otherwise, apply the induction hypothesis a second time to $Z$.

  Finally, suppose that the cut has the form $t \bowtie_X s$, where
  $s=(w_1 + w_2 + \cdots + w_m)$.  Since $\ll$ is a partial order on
  the nodes $F$, at least one of these witnesses will be
  $\ll$-maximal.  Suppose, without loss of generality, that $w_1$ is
  $\ll$-maximal among the $w_i$'s.  Then apply the duplication
  reduction to the kingdom of $X$, with the decomposition $s=
  (w_1)+(w_2+\dots +w_m)$; we obtain an $AC$ forest $G', t^L \bowtie_Y
  (w_1), t^R \bowtie_Z (w_2 + \dots+w_m)$. Now apply the induction
  hypothesis to $Z$, to obtain an $AC$ forest $G'', t' \bowtie
  (w'_1)$: crucially, since $w_1$ was not in $k(Z)$, the positive
  width of this cut does not change after $Z$ is eliminated.  We may
  thus apply the induction hypothesis again to complete the proof.
\end{proof}

\begin{theorem}[Cut elimination]
\label{thm:cut-elim}
  If $F$ is an expansion net with type $\G$, there is a cut-free
  expansion net $F'$, reachable by the cut-reduction operations from
  $F$, such that the type $\D$ of $F'$ is a subsequent of $\G$.
\end{theorem}
\begin{proof}
  By successive applications of the principal lemma, we can remove all
  the cuts from $F$, the result being an AC typed forest whose type
  $\Delta$ is a closure of a subsequent of the type of $F$: since
  the closure of a classical sequent is just the sequent itself,
  $\Delta$ is a subsequent of the type of $F$, and so $F'$ is an
  expansion-net.
\end{proof}

\section{Conclusion}
\label{sec:conclusion}

Expansion-nets provide a class of abstract proof objects for classical
propositional logic which satisfy our checklist of good
properties. There is a sequent calculus ($\LKvar$) with a canonical
function from proofs in that calculus to expansion-nets (given in
Definition~\ref{def:exnets}) There is a correctness criterion
(Definition~\ref{def:correctness}) which can be checked in polynomial
time, such that the correct proof structures are precisely the
expansion nets.  We have sequentialization into $\LKvar$
(Theorem~\ref{thm:seq}), and weakly normalizing cut-elimination
directly on expansion-nets (Theorem~\ref{thm:cut-elim}).  The last two
of these results are new to the paper (although the former was
sketched in~\cite{McK10exnets}); their proofs rely on the
characterization of subnets of expansion nets, including the new
notion of contiguous subnet defined in this paper. In addition to these
properties, expansion-nets also identify a more natural set of sequent
derivations than do the previously existing notions of abstract proof.

We mention some further directions:

\subsubsection*{Beyond propositional logic}

The terminology $expansion$ deliberately recalls
Miller~\cite{Mil87ComRep}, whose \emph{expansion tree proofs} can be
seen as a prototype notion of proof-net for classical logic.  The
paper~\cite{McK11ProNeHer} makes this connection explicit in the case of
first-order prenex formulae; the paper introduces a notion of
\emph{Herbrand net} using Girard's notion of a quantifier jump, in
which provability at the propositional level is treated as trivial ---
propositional axioms are replaced by arbitrary propositional
tautologies. Expansion-tree proofs themselves do not provide a good
notion of proof-net when we move beyond sequents of prenex formulae:
they lack the fine-grained propositional structure of expansion-nets
and so do not seem to have well-behaved cut-elimination.  However, we
foresee no major obstacles in combining Herbrand nets with the results
of the current paper to capture nets for first- or higher-order
classical quantifiers, including cut-elimination.

\subsubsection*{Nets for additively formulated classical logic}
The correctness/sequentialization results for our nets are heavily
tied to the multiplicatively formulated sequent calculus.  It is, of
course, possible to extract an ed-net from a proof in an additively
formulated calculus, but there are natural identities in those calculi
which are not validated by our nets. Taking the view that the additive
classical connectives are essentially different operations (that happen
to coincide at the level of provability), we look for natural notions
of proof net for additively formulated classical logic.

\bibliography{nets}

\begin{thebibliography}{10}

\bibitem{Bel97SubMix}
Gianluigi Bellin.
\newblock Subnets of proof-nets in multiplicative linear logic with {MIX}.
\newblock {\em Mathematical Structures in Computer Science}, 7(6):663--699,
  1997.

\bibitem{Beletal06CatProThe}
Gianluigi Bellin, Martin Hyland, Edmund Robinson, and Christian Urban.
\newblock Categorical proof theory of classical propositional calculus.
\newblock {\em Theor. Comput. Sci.}, 364(2):146--165, 2006.

\bibitem{212898}
Gianluigi. Bellin and Jacques. van~de Wiele.
\newblock Subnets of proof-nets in {MLL-}.
\newblock In {\em Proceedings of the workshop on Advances in linear logic},
  pages 249--270, New York, NY, USA, 1995. Cambridge University Press.

\bibitem{BruTh}
Kai Br\"unnler.
\newblock {\em Deep Inference and Symmetry in Classical Proofs}.
\newblock PhD thesis, Technische Universit{\"a}t Dresden, 2003.

\bibitem{BruLCL06}
Kai {Br\"unnler}.
\newblock Locality for classical logic.
\newblock {\em Notre Dame Journal of Formal Logic}, 47:557--580, 2006.

\bibitem{CookReckhow79PropProo}
Stephen~A. Cook, Robert, and A.~Reckhow.
\newblock The relative efficiency of propositional proof systems.
\newblock {\em Journal of Symbolic Logic}, 44:36--50, 1979.

\bibitem{DanReg89StrMul}
V.~Danos and L.~Regnier.
\newblock The structure of multiplicatives.
\newblock {\em Archive for Mathematical Logic}, 28:181--203, 1989.

\bibitem{Dan90Thesis}
Vincent Danos.
\newblock {\em La logique linéaire appliquée à l’étude de divers
  processus de normalisation et principalement du lambda calcul.}
\newblock PhD thesis, Univ. de Paris, 1990.

\bibitem{LinLogPrimer}
Vincent Danos and Roberto di~Cosmo.
\newblock The linear logic primer.
\newblock Available at \url{http://www.dicosmo.org/CourseNotes/LinLog/}.

\bibitem{Fuhr06OrdEnr}
Carsten F{\"u}hrmann and David Pym.
\newblock Order-enriched categorical models of the classical sequent calculus.
\newblock {\em Journal of Pure and Applied Algebra}, 204(1):21 -- 78, 2006.

\bibitem{Gentzen34}
Gerhard Gentzen.
\newblock Untersuchungen {\"u}ber das logische {S}chlie{\ss}en.
\newblock {\em Mathematische Zeitschrift}, 39:176--210, 405--431, 1934.

\bibitem{Girard87Linear}
Jean-Yves Girard.
\newblock Linear logic.
\newblock {\em Theor. Comput. Sci.}, 50:1--102, 1987.

\bibitem{Girard91NewCon}
Jean-Yves Girard.
\newblock A new constructive logic: Classical logic.
\newblock {\em Mathematical Structures in Computer Science}, 1(3):255--296,
  1991.

\bibitem{Gir96ProNetPar}
Jean-Yves Girard.
\newblock Proof-nets: The parallel syntax for proof-theory.
\newblock In {\em Logic and Algebra}, pages 97--124. Marcel Dekker, 1996.

\bibitem{GugGunStra10BreFlows}
Alessio Guglielmi, Tom Gundersen, and Lutz Stra{\ss}burger.
\newblock Breaking paths in atomic flows for classical logic.
\newblock In {\em LICS}, pages 284--293. IEEE Computer Society, 2010.

\bibitem{DJDHughes06TowHilbProb}
Dominic J.~D. Hughes.
\newblock Towards {H}ilbert's 24th problem: Combinatorial proof invariants.
\newblock {\em Electron. Notes Theor. Comput. Sci.}, 165:37--63, 2006.

\bibitem{DJDHughes06PWS}
Dominic~J.D. Hughes.
\newblock Proofs {W}ithout {S}yntax.
\newblock {\em Annals of {M}athematics}, 143(3):1065--1076, November 2006.

\bibitem{DJDHughes:hybrid}
Dominic~J.D. Hughes.
\newblock A minimal classical sequent calculus free of structural rules.
\newblock Archived as math.LO/0506463 at arXiv.org, July 2010.

\bibitem{LamStra05NamProCla}
F~Lamarche and L~Strassburger.
\newblock Naming proofs in classical logic.
\newblock In {\em Proceedings of TLCA '05}. Springer-Verlag, 2005.

\bibitem{LamStr05ConsFreeBoo}
Francois Lamarche and Lutz Strassburger.
\newblock Constructing free boolean categories.
\newblock In {\em LICS '05: Proceedings of the 20th Annual IEEE Symposium on
  Logic in Computer Science}, pages 209---218, Washington, DC, USA, 2005. IEEE
  Computer Society.

\bibitem{LinWin92constant-only}
Patrick Lincoln and Timothy Winkler.
\newblock Constant-only multiplicative linear logic is {NP}-complete.
\newblock {\em Theoretical Computer Science}, 135:135--155, 1992.

\bibitem{McK11ProNeHer}
Richard McKinley.
\newblock Proof nets for {H}erbrands theorem.
\newblock Accepted for publication, \textit{ACM Transactions on Computational
  Logic}.

\bibitem{McK10exnets}
Richard McKinley.
\newblock Expansion nets: proof-nets for propositional classical logic.
\newblock In {\em Proceedings of the 17th international conference on Logic for
  programming, artificial intelligence, and reasoning}, LPAR'10, pages
  535--549, Berlin, Heidelberg, 2010. Springer-Verlag.

\bibitem{Mil87ComRep}
Dale Miller.
\newblock A compact representation of proofs.
\newblock {\em Studia Logica}, 46(4):347--370, 1987.

\bibitem{Rob03ProNetCla}
Edmund Robinson.
\newblock Proof nets for classical logic.
\newblock {\em Journal of Logic and Computation}, 13(5):777--797, 2003.

\bibitem{Thiele01hilbert}
R\"udiger Thiele.
\newblock Hilbert's twenty-fourth problem.
\newblock {\em American Mathematical Monthly}, 110:2003, 2001.

\bibitem{Trim94Thesis}
Todd Trimble.
\newblock {\em Linear logic, bimodules, and full coherence for autonomous
  categories.}
\newblock PhD thesis, Rutgers University, 1994.

\end{thebibliography}

\end{document}